\documentclass[11pt]{amsart}
\usepackage[dvipsnames]{xcolor}
\usepackage[colorlinks=true, linkcolor=NavyBlue, citecolor=NavyBlue, pdfpagelayout=OneColumn, pdfstartview={XYZ null null 0.75}]{hyperref}
\usepackage{mathtools, stmaryrd, subfigure, tikz}
\numberwithin{equation}{section}

\def\F{\mathbb F}
\def\N{\mathbb N}
\def\Z{\mathbb Z}
\def\cC{\mathcal C}
\def\cF{\mathcal F}
\def\CFKi{CFK^\infty}
\def\fl{\textup{fl}}
\def\horz{\textup{horz}}
\def\vert{\textup{vert}}
\def\red{\textup{red}}
\def\cs{\mathrel\#}

\newtheorem{Theorem}{Theorem}
\newtheorem{theorem}{Theorem}[section]
\newtheorem{corollary}[theorem]{Corollary}
\newtheorem{lemma}[theorem]{Lemma}
\newtheorem{proposition}[theorem]{Proposition}
\theoremstyle{definition}
\newtheorem{definition}[theorem]{Definition}
\theoremstyle{remark}
\newtheorem{remark}[theorem]{Remark}
\newtheorem*{ack}{Acknowledgments}
\newtheorem*{org}{Organization}

\makeatletter
\newcommand*\numberthis{
\addtocounter{equation}{1}
\tag\theequation }
\renewcommand*\env@cases[1][1.2]{
\let\@ifnextchar\new@ifnextchar
\left\lbrace
\def\arraystretch{#1}
\array{@{}l@{\quad}l@{}} }
\makeatother

\title{On the knot Floer filtration of the concordance group}
\author{Stephen Hancock}
\address{}
\email{ssh2127@columbia.edu}
\author{Jennifer Hom}
\address {Department of Mathematics, Columbia University, New York, NY 10027}
\email{hom@math.columbia.edu}
\author{Michael Newman}
\address{Department of Mathematics, University of Michigan, Ann Arbor, MI 48109}
\email{mgnewman@umich.edu}

\begin{document}

\begin{abstract}
The knot Floer complex together with the associated concordance invariant $\varepsilon$ can be used to define a filtration on the smooth concordance group. We show that the indexing set of this filtration contains $\N \times \Z$ as an ordered subset.
\end{abstract}
\maketitle

\section{Introduction}

Two knots in $S^3$ are called \emph{concordant} if they cobound a smooth, properly embedded cylinder in $S^3 \times [0,1]$. The set of knots in $S^3$, modulo concordance, forms an abelian group, the \emph{concordance group}, denoted $\cC$, where the operation is induced by connected sum. If a knot is concordant to the unknot, then we say that it is \emph{slice}. The inverse of a knot $K$ is given by $-K$, the reverse of the mirror of $K$. It is straightforward to show that $K_1$ and $K_2$ are concordant if and only if $K_1 \cs -K_2$ is slice.

A powerful tool for understanding knots is the knot Floer complex, defined by Ozsv\'ath and Szab\'o \cite{OSknot}, and independently Rasmussen \cite{Rknot}. To a knot $K$, they associate a $\Z \oplus \Z$-filtered chain complex, denoted $\CFKi(K)$, whose filtered chain homotopy type is an invariant of $K$. Associated to the complex $\CFKi(K)$ is a $\{-1, 0, 1\}$-valued concordance invariant $\varepsilon(K)$ defined in \cite{Homsmooth}. The set of such filtered chain complexes forms a monoid under the operation of tensor product, and modulo an equivalence relation defined in terms of $\varepsilon$, this monoid can be made into a group, denoted $\cF$.

The advantage of this approach is that there is a homomorphism from
\[ \cC \to \cF, \]
defined by $[K] \mapsto [\CFKi(K)]$. Moreover, the group $\cF$ has a rich algebraic structure coming from a total ordering. This ordering gives a filtration on $\cF$ that can be pulled back to a filtration on $\cC$, called the \emph{knot Floer filtration}. While the indexing set of the knot Floer filtration is largely unknown, our main theorem gives a lower bound on the complexity of this indexing set.

\begin{Theorem} \label{thm:ordertype}
The indexing set of the knot Floer filtration contains a subset that is order isomorphic to $\N \times \Z$. Specifically, we can index the filtration by
\[ S = \{(i, j) \mid (i, j) \geq (0, 0)\}, \]
where $S \subset \Z \times \Z$ inherits the lexicographical ordering. Furthermore, each successive quotient is infinite, i.e., for $(i, j), (i'\!, j') \in S$ and $(i, j) < (i'\!, j')$, we have that
\[ \Z \subset \cF_{(i'\!, j')}/\cF_{(i, j)}. \]
\end{Theorem}

Filtrations have been shown to be an effective tool for studying the concordance group. For example, Cochran, Orr, and Teichner \cite{COT} define the $n$-solvable filtration
\[ \cdots \subset \cF_{-(n+1)} \subset \cF_{-n.5} \subset \cF_{-n} \subset \cdots \subset \cF_{-1} \subset \cF_{-0.5} \subset \cF_0 \subset \cC, \]
a filtration indexed by negative half integers. (Note that we adopt the convention that an indexing set for a filtration $\cF$ is an ordered set $S$ with the property that for $a, b \in S$, $a < b$ implies that $\cF_a \subset \cF_b$.) It was shown by Cochran, Harvey, and Leidy \cite{CHL} that the quotient $\cF_{-n}/\cF_{-n.5}$ is of infinite rank for each non-negative integer $n$. Recent work of Cochran, Harvey, and Horn \cite{CHH} defines the bipolar filtration, again indexed by the negative natural numbers, and they also show that each successive quotient is of infinite rank. Our approach to filtering the concordance group utilizes a finer indexing set than the filtrations of \cite{COT} and \cite{CHH}.

The proof of our result requires the computation of a large family of knot Floer complexes, modulo ``$\varepsilon$-equivalence". While computing the knot Floer complex in general is difficult, we use two properties of knot Floer homology that give us a large class of knots for which the computation simplifies drastically.

Recall that an \emph{$L$-space} is a rational homology $S^3$ for which $\text{rk }\widehat{HF}(Y) = |H_1(Y; \Z)|$, so named because this class of $3$-manifold includes lens spaces. The first property that we use pertains to a family of knots called \emph{$L$-space knots}, that is, knots which admit a positive $L$-space surgery. It is well-known that positive torus knots admit positive lens space surgeries and thus are $L$-space knots. Ozsv\'ath and Szab\'o \cite[Theorem 1.2]{OSlens} show that the knot Floer complex of an $L$-space knot $K$ is completely determined by the Alexander polynomial of $K$. Moreover, Hedden \cite{Heddencable} proves that sufficiently large cables of $L$-space knots are again $L$-space knots. Thus, to understand the knot Floer complexes of torus knots and appropriate cables, it is sufficient to know the knot's Alexander polynomial. It is well-known that the Alexander polynomial of $T_{p,q}$, the $(p, q)$-torus knot, is
\[ \Delta_{T_{p,q}}(t) = \frac{(t^{pq} - 1)(t - 1)}{(t^p - 1)(t^q - 1)} \]
and that the Alexander polynomial of $K_{p,q}$, the $(p, q)$-cable of $K$, is
\[ \Delta_{K_{p,q}}(t) = \Delta_K(t^p) \cdot \Delta_{T_{p,q}}(t), \]
where $p$ denotes the longitudinal winding and $q$ the meridional winding.

The second useful property concerns the behavior of these invariants under basic topological operations. Let $\CFKi(K)^*$ denote the dual of $\CFKi(K)$; we give the precise definition of the dual complex in Section \ref{sec:background}. Ozsv\'ath and Szab\'o \cite{OSknot} show that
\[ \CFKi(K_1 \cs K_2) \simeq \CFKi(K_1) \otimes \CFKi(K_2) \]
and that
\[ \CFKi(-K) \simeq \CFKi(K)^*, \]
allowing us to compute $\CFKi$ for linear combinations of $L$-space knots and their inverses.

With these techniques, we are able to realize a large class of $\Z \oplus \Z$-filtered chain complexes (up to $\varepsilon$-equivalence), and by studying the structure of the group $\cF$, we can understand where in the filtration these knots lie.

The order type in Theorem \ref{thm:ordertype} is almost certainly not a complete description of the indexing set of $\cF$. One reason for this is that we limited ourselves to linear combinations of a small class of $L$-space knots for computational reasons. Moreover, to achieve our result, we needed only to consider connected sums of at most two $L$-space knots, and the $L$-space knots in question were always cables of torus knots. Further work suggests that with linear combinations of iterated torus knots, a richer order type is possible. An interesting question to consider is whether linear combinations of non-$L$-space knots would further enlarge the order type.

The results of \cite{Homsmooth} defined various numerical concordance invariants associated to $\CFKi(K)$ that, in a sense, are a refinement of the Ozsv\'ath-Szab\'o $\tau$ invariant \cite{OS4ball}. This paper studies such invariants in more depth, giving a better understanding of the relationship between these invariants and the structure of the concordance group.

\begin{org}
We begin in Section \ref{sec:background} with the necessary background on knot Floer homology, totally ordered groups, and $L$-space knots, including definitions of the invariant $\varepsilon$ and the group $\cF$. We proceed to prove algebraic results about $\cF$ (Sections \ref{sec:tensor} and \ref{sec:ordering}) and the existence of a certain family of elements in $\cF$ (Section \ref{sec:floer}) through direct computation. In Section \ref{sec:archimedean}, we find knots that allow us to apply our preceding lemmas to understand the order type of $\cF$, which leads to the proof of Theorem \ref{thm:ordertype}. We work with coefficients in $\F = \Z/2\Z$ throughout.
\end{org}

\begin{ack}
The ideas for this paper began during the Summer 2011 Topology REU at Columbia University, which was partially funded by NSF grant DMS-0739392. The second author was partially supported by NSF grant DMS-1307879. The authors would like to thank the other participants in the group, Vivian Josie Bailey and Chun Ye, for their interest in the project, and the organizers of the program for providing the opportunity to work together. The authors would also like to thank the referee for many helpful suggestions.
\end{ack}

\section{Background} \label{sec:background}

\subsection{The knot Floer complex and concordance}
We begin with the necessary background on knot Floer homology, as defined in \cite{OSknot} and \cite{Rknot}. To a knot $K$, we associate a $\Z \oplus \Z$-filtered, $\Z$-graded chain complex over $\F[U, U^{-1}]$, where $U$ is a formal variable. The $\Z$-grading is called the Maslov, or homological, grading. We denote this complex by $\CFKi(K)$, and the filtered chain homotopy type of $\CFKi(K)$ is an invariant of the knot $K$. The ordering on $\Z \oplus \Z$ is given by $(i, j) \leq (i'\!, j')$ if $i \leq i'$ and $j \leq j'$.

The differential, $\partial$, decreases the homological grading by one and respects the $\Z \oplus \Z$-filtration. Multiplication by $U$ shifts the $\Z$-grading by two and decreases the $\Z \oplus \Z$-filtration by $(1, 1)$. Connected sum of knots corresponds to tensor product of their respective chain complexes. That is,
\[ \CFKi(K_1 \cs K_2) \simeq \CFKi(K_1) \otimes_{\F[U, U^{-1}]} \CFKi(K_2). \]
Taking the reverse of the mirror image of a knot corresponds to taking the dual of its knot Floer complex. That is,
\[ \CFKi(-K) \simeq \CFKi(K)^*, \]
where $\CFKi(K)^*$ denotes the dual of $\CFKi(K)$, i.e., $\text{Hom}_{\F[U, U^{-1}]}(\CFKi(K), \F[U, U^{-1}])$. The complex $\CFKi(K)$ is filtered chain homotopic to the complex obtained by interchanging $i$ and $j$.

A basis $\{x_k\}$ over $\F[U, U^{-1}]$ for a filtered chain complex $C$ is a \emph{filtered basis} if $\{U^n \cdot x_k \mid U^n \cdot x_k \in C_{i,j},\, n \in \Z\}$ is a basis over $\F$ for the subcomplex $C_{i,j}$ for all $i, j \in \Z$, where $C_{i,j}$ denotes the $(i, j)^\text{th}$-filtered subcomplex. In this paper, we will often perform a filtered change of basis, producing a new filtered basis from an old one. Given a filtered basis $\{x_k\}$, we can produce a new filtered basis $\{x'_k\}$, where
\[ x'_k = \begin{cases}[1] x_k + x_\ell &\text{if } k = n \\ x_k &\text{otherwise} \end{cases} \]
for some $n$ and $\ell$ such that the filtration level of $x_\ell$ is less than or equal to that of $x_n$. In other words, one may replace a basis element with itself plus elements of lesser or equal filtration level. We will often omit the prime from the new basis and denote this change of basis by
\[ x_n \to x_n + x_\ell, \qquad x_k \to x_k,\  k \neq n. \]

To better understand $\Z \oplus \Z$-filtered chain complexes, it is convenient to depict them in the $(i, j)$-plane, where the $(i, j)$-coordinates depict the filtration level. The $\Z$-grading is suppressed from this picture. We consider the generators over $\F[U, U^{-1}]$. We place an element $U^n \cdot x$ at the lattice point $(i-n, j-n)$, where $(i, j)$ is the filtration level of $x$. We use arrows to describe the differential; if $U^n \cdot y$ appears with non-zero coefficient in $\partial x$, then we draw an arrow from $x$ to $U^n \cdot y$. Since the differential respects the $\Z \oplus \Z$-filtration, the arrows will necessarily point (non-strictly) to the left and down. Up to filtered chain homotopy, one may assume that the differential will strictly decrease the filtration \cite[Lemma 4.5]{Rknot}, and indeed, that will be the case for all of the complexes we consider. Moreover, for all of the complexes under consideration in this paper, there exists a basis where each arrow connects elements of the same $U$-degree. This is because the knots in this paper are all linear combinations of $L$-space knots, which are described in Section \ref{sec:Lspace}. Thus, it is sufficient to consider a single copy of each generator, rather than all of the $U$-translates. At times, it will be convenient to consider only the part of $\partial$ that preserves the $j$- or $i$-filtration level. We use $\partial^\horz$ and $\partial^\vert$, respectively, to denote these.

The subquotient of $\CFKi(K)$ consisting of the $i = 0$ column yields the complex $\widehat{CF}(S^3)$, and so the homology of the $i = 0$ column (or in fact, any column, up to a grading shift) is isomorphic to $\F$. Similarly, the homology of any row is also isomorphic to $\F$.

The picture for $\CFKi(K)^*$ is closely related to the picture for $\CFKi(K)$; one simply reverses the direction of each arrow, as well as both filtrations. (In practice, this may be accomplished by turning the page upside down and reversing the directions of all of the arrows.)

A basis $\{x_i\}$ over $\F[U, U^{-1}]$ for $\CFKi(K)$ is called \emph{vertically simplified} if for each basis element $x_i$, exactly one of the following holds:
\begin{itemize}
	\item There is a unique incoming vertical arrow into $x_i$.
	\item There is a unique outgoing vertical arrow from $x_i$.
	\item There are no vertical arrows entering or leaving $x_i$.
\end{itemize}
Note that since the homology of a column is $\F$, there is a unique basis element of a vertically simplified basis with no incoming or outgoing vertical arrows, called the \emph{vertically distinguished element}. The analogous definition can be made for a \emph{horizontally simplified} basis. By \cite[Proposition 11.52]{LOT}, one may always choose a basis which is vertically simplified, or if one prefers, horizontally simplified.

Given a vertically simplified basis, consider the subquotient complex associated to the $i = 0$ column. The $j$-coordinate of the vertically distinguished element in this column is a concordance invariant, defined by Ozsv\'ath and Szab\'o in \cite{OS4ball} and denoted $\tau(K)$.

While it remains unknown whether a simultaneously vertically and horizontally simplified basis always exists in general, we are able to find such a basis for the complexes under consideration in this paper. Moreover, one may always find a horizontally simplified basis where one of the basis elements, say $x_0$, is the distinguished element of some \emph{vertically} simplified basis \cite[Lemmas 3.2 and 3.3]{Homcable}. The $\{-1, 0, 1\}$-valued concordance invariant $\varepsilon$ can be defined in terms of such a basis.

\begin{definition} \label{def:epsilon}
The invariant $\varepsilon(K)$ is defined in terms of the above basis for $\CFKi(K)$ as follows:
\begin{enumerate}
	\item $\varepsilon(K) = 1$ if there is a unique incoming horizontal arrow into $x_0$.
	\item $\varepsilon(K) = -1$ if there is a unique outgoing horizontal arrow from $x_0$.
	\item $\varepsilon(K) = 0$ if there are no horizontal arrows entering or leaving $x_0$.
\end{enumerate}
\end{definition}

\noindent To emphasize that $\varepsilon$ is in fact an invariant of a bifiltered chain complex, we may at times write $\varepsilon(\CFKi(K))$, rather than $\varepsilon(K)$. Alternatively, the invariant $\varepsilon$ can be defined in terms of the (non-)vanishing of certain cobordism maps on $\widehat{HF}$, as in \cite[Definition 3.1]{Homsmooth}.

\begin{proposition}[{\cite[Proposition 3.6]{Homcable}}]
The following are properties of $\varepsilon(K)$:
\begin{enumerate}
	\item If $K$ is smoothly slice, then $\varepsilon(K) = 0$.
	\vspace{5pt}
	\item $\varepsilon(-K) = -\varepsilon(K)$.
	\vspace{5pt}
	\item	\begin{enumerate}
		\item If $\varepsilon(K) = \varepsilon(K')$, then $\varepsilon(K \cs K') = \varepsilon(K) = \varepsilon(K')$.
		\item If $\varepsilon(K) = 0$, then $\varepsilon(K \cs K') = \varepsilon(K')$.
		\end{enumerate}
\end{enumerate}
\end{proposition}

\noindent Notice that if $K_1$ and $K_2$ are concordant, then $\varepsilon(\CFKi(K_1) \otimes \CFKi(K_2)^*) = 0$, motivating the following definition.

\begin{definition}
Two bifiltered chain complexes $C_1$ and $C_2$ are \emph{$\varepsilon$-equivalent}, denoted $\sim_\varepsilon$, if
\[ \varepsilon(C_1 \otimes C_2^*) = 0. \]
\end{definition}

Recall that the concordance group $\mathcal{C}$ is obtained as a quotient of the monoid of knots under connected sum by the equivalence relation of concordance. In a similar manner, chain complexes under tensor product form a monoid, and using the idea of $\varepsilon$-equivalence, we can obtain a group.

Consider the monoid $(M, \otimes_{\F[U, U^{-1}]})$ of bifiltered chain complexes up to filtered chain homotopy such that:
\begin{itemize}
	\item The underlying module is a free $\F[U, U^{-1}]$ module.
	\item The total homology of the complex is isomorphic to $\F[U, U^{-1}]$.
	\item The vertical homology of the complex is isomorphic to $\F[U, U^{-1}]$.
	\item The complex obtained by interchanging $i$ and $j$ is filtered chain homotopic to the original complex.
\end{itemize}

\begin{definition}
The group $\cF_\textup{alg}$ is defined to be
\[ \cF_\textup{alg} = \big( M, \otimes \big)/\sim_\varepsilon. \]
\end{definition}

\noindent We denote the group operation of $\cF_\text{alg}$ by $+$ and the identity by $0$. We may also consider the subgroup of $\cF_\text{alg}$ generated by complexes that are realized as $\CFKi(K)$ for some knot $K \subset S^3$.

\begin{definition}
The group $\cF$ is
\[ \cF = \big( \{\CFKi(K) \mid K \subset S^3\}, \otimes \big)/\sim_\varepsilon. \]
\end{definition}

\noindent Clearly $\cF \subseteq \cF_\text{alg}$ since $\{\CFKi(K)\} \subseteq M$. It is known that $\{\CFKi(K)\} \neq M$, since there does not exist an $L$-space knot $K$ with $a_1(K) > 1$ \cite[Theorem 2.3]{Rlens}; see Section \ref{sec:notation} below for the definition of $a_1$. However, it is an open question whether $\cF = \cF_\text{alg}$.

It is clear from the definition of $\cF$ and properties of $\CFKi$ that we obtain a group homomorphism
\[ \cC \to \cF \]
by sending $[K]$ to $[\CFKi(K)]$. Calling this map $\phi$, notice that $\cF \cong \cC/\text{ker }\phi = \cC/\{[K] \mid \varepsilon(K) = 0\}$. For ease of notation, we write
\[ \llbracket K \rrbracket \]
to denote $[\CFKi(K)]$. Note that $-\llbracket K \rrbracket = \llbracket -K \rrbracket$ and $\llbracket \text{unknot} \rrbracket = 0$.

One of the advantages of this approach is that the group $\cF$ has a rich algebraic structure. In particular, $\cF$ is totally ordered, with the ordering given by
\[ [\CFKi(K_1)] > [\CFKi(K_2)] \iff \varepsilon(\CFKi(K_1) \otimes \CFKi(K_2)^*) = 1. \]
By considering the behavior of $\varepsilon$ under connected sum, it follows that this total ordering is well-defined.

\subsection{Totally ordered groups}
Two totally ordered sets $S_1$ and $S_2$ are \emph{order isomorphic} if there exists a bijection $S_1 \to S_2$ such that both the bijection and its inverse are order-preserving. The order equivalence class of $S$ is called the \emph{order type} of $S$.

Given a totally ordered abelian group $G$, one can naturally define a notion of absolute value, i.e., for any $g \in G$,
\[ |g| = \begin{cases}[1] g &\text{if } g \geq \text{id}_G \\ -g &\text{otherwise}. \end{cases} \]
Two elements $g$ and $h$ of a totally ordered abelian group $G$ are said to be \emph{Archimedean equivalent}, denoted $\sim_\text{Ar}$, if there exist $m, n \in \N$ such that
\[ m \cdot |g| > |h| \qquad \text{and} \qquad n \cdot |h| > |g|. \]
The set of Archimedean equivalence classes of $G$ inherits an ordering from the group, and the order type of this set is called the \emph{coarse order type} of the group.

Let $[g]_\text{Ar}$ denote the Archimedean equivalence class of $g$. If $[h]_\text{Ar} < [g]_\text{Ar}$, then $n \cdot |h| < |g|$ for all $n \in \N$, and we write
\[ |h| \ll |g|. \]
(If one restricts oneself to only positive elements in the group, then the absolute value signs may be omitted.) For positive $\llbracket K \rrbracket, \llbracket J \rrbracket \in \cF$, note that $\llbracket K \rrbracket \gg \llbracket J \rrbracket \iff \varepsilon(K \cs -nJ) = 1$ for all $n \in \N$.

A totally ordered group inherits a natural filtration, with the indexing set given by the coarse order type of the group. Given an Archimedean equivalence class, choose a representative $g$, and consider the subgroup
\[ H_g = \{h \in G \mid [h]_\text{Ar} \leq [g]_\text{Ar}\}. \]
Indeed, it follows from the definition of Archimedean equivalence that the set $H_g$ is closed under the group operation, and it is clear that if $h \in H_g$, then the inverse of $h$ is as well. The filtration now also follows from the definition of Archimedean equivalence, since $[g_1]_\text{Ar} < [g_2]_\text{Ar}$ implies that $H_{g_1} \subset H_{g_2}$ and that $\Z \subset H_{g_2}/H_{g_1}$, generated by $g_2$.

Applying these tools to the group $\cF$, we obtain a filtration on $\cF$, which we may pull back to give a filtration on $\cC$. The effectiveness of this approach is largely determined by the coarse order type of $\cF$. For instance, the second author showed \cite[Proposition 4.8]{Homsmooth} that the coarse order type of $\cF$ contains $\omega$ as an ordered subset, with $\cF_n := \phi^{-1}\big[ H_{\llbracket T_{n,n+1} \rrbracket} \big]$ giving a filtration on $\cC$ indexed by $\N$. This is precisely the reversed order type of the $n$-solvable and bipolar filtrations. Our goal is to achieve an indexing set with finer order type.

\subsection{$L$-space knots} \label{sec:Lspace}
One of our main tools for computing the knot Floer complex of large families of knots concerns $L$-space knots. Recall that an \emph{$L$-space} $Y$ is a rational homology sphere for which $\text{rk }\widehat{HF}(Y) = |H_1(Y; \Z)|$, and that an \emph{$L$-space knot} is a knot on which some positive surgery is an $L$-space. In \cite[Theorem 1.2]{OSlens}, Ozsv\'ath and Szab\'o show that if a knot $K$ admits a positive $L$-space surgery, then its knot Floer complex is completely determined by the Alexander polynomial of $K$. In particular, if $K$ is an $L$-space knot, then the Alexander polynomial of $K$ is of the form
\[ \Delta_K(t) = \sum_{i=0}^{2m} (-1)^it^{n_i} \]
for a positive integer $m$ and some strictly increasing sequence of $n_i \in \Z_{\geq 0}$ satisfying the symmetry requirement that
\[ n_i + n_{2m-i} = 2g(K), \]
where $g(K)$ is the genus of $K$ and we have normalized the Alexander polynomial to have a constant term and no negative exponents, i.e., $n_0 = 0$.

The sequence of $n_i$ determines the knot Floer complex of $K$. A filtered basis over $\F[U, U^{-1}]$ for $\CFKi(K)$ is given by $\{x_i\}$, $i = 0, \ldots, 2m$, with the following differentials:
\[ \partial x_i = \begin{cases}[1] x_{i-1} + x_{i+1} &i \text{ odd} \\ 0 &i \text{ even}, \end{cases} \]
where the arrow from $x_i$ to $x_{i-1}$ is horizontal of length $n_i-n_{i-1}$, and the arrow from $x_i$ to $x_{i+1}$ is vertical of length $n_{i+1}-n_i$. See Figure \ref{fig:T_3,4} for an example.

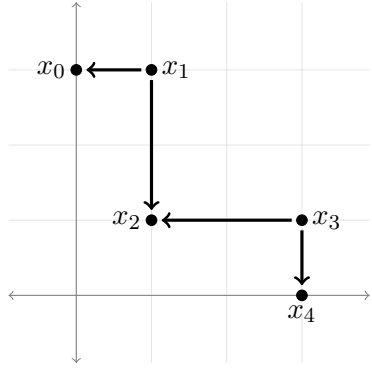
\begin{figure}[htb]
\centering
\begin{tikzpicture}
	\draw [step=1, black!10!white, very thin] (-0.9, -0.9) grid (3.9, 3.9);
	\draw [gray, thin, <->] (-0.9, 0) -- (3.9, 0);
	\draw [gray, thin, <->] (0, -0.9) -- (0, 3.9);
	\filldraw (0, 3) circle (2pt) node[] (x0){};
	\filldraw (1, 3) circle (2pt) node[] (x1){};
	\filldraw (1, 1) circle (2pt) node[] (x2){};
	\filldraw (3, 1) circle (2pt) node[] (x3){};
	\filldraw (3, 0) circle (2pt) node[] (x4){};
	\draw [very thick, <-] (x0) -- (x1);
	\draw [very thick, <-] (x2) -- (x1);
	\draw [very thick, <-] (x2) -- (x3);
	\draw [very thick, <-] (x4) -- (x3);
	\node [left] at (x0) {$x_0$};
	\node [right] at (x1) {$x_1$};
	\node [left] at (x2) {$x_2$};
	\node [right] at (x3) {$x_3$};
	\node [below] at (x4) {$x_4$};
\end{tikzpicture}
\caption{A basis for $\CFKi(T_{3,4})$, where $T_{3,4}$ denotes the $(3, 4)$-torus knot. The Alexander polynomial of $T_{3,4}$ is $\Delta_{T_{3,4}}(t) = 1-t+t^3-t^5+t^6$.}
\label{fig:T_3,4} \end{figure}

In the next section, we introduce special notation for denoting certain families of bifiltered chain complexes, with the above result taking the form
\begin{equation} \label{eqn:staircase}
\llbracket K \rrbracket = [(n_i - n_{i-1})_{i=1}^m]
\end{equation}
in that notation. Notice that in the $(i, j)$-plane, this complex has the appearance of a ``staircase". (In fact, having such a staircase complex is a necessary and sufficient condition for being an $L$-space knot \cite{OSlens}.) Such complexes will play a key role in this paper.

\subsection{Notation for chain complex classes} \label{sec:notation}
In this section, we define some notation that will be useful to describe the complexes of interest.

\begin{definition} \label{def:Ctype}
A bifiltered complex $C$ that represents an element in $\cF_\textup{alg}$ is of type $(a_1, \ldots, a_m)$ if it is doubly-filtered chain homotopy equivalent to a direct sum
\[ C_\red \oplus C_A, \]
where $C_A$ is acyclic, $C_\red$ has no acyclic summands, and $C_\red$ admits a simultaneously vertically and horizontally simplified basis $\{x_0, \ldots, x_{2m}\}$ with the following vertical and horizontal differentials:
\[ \partial^\horz x_i = \begin{cases}[1] x_{i-1} &i \textup{ odd},\, a_i > 0 \\ x_{i+1} &i \textup{ even},\, i \neq 2m,\, a_{i+1} < 0 \\ 0 &\textup{otherwise} \end{cases} \qquad \text{and} \qquad \partial^\vert x_i = \begin{cases}[1] x_{i+1} &i \textup{ odd},\, a_{i+1} > 0 \\ x_{i-1} &i \textup{ even},\, i \neq 0,\, a_i < 0 \\ 0 &\textup{otherwise}, \end{cases} \]
where the arrow between $x_i$ and $x_{i-1}$ is of length $|a_i|$ and $a_i = a_{2m+1-i}$ for $i = m+1, \ldots, 2m$.
\end{definition}

\begin{remark}
In the above definition, the differentials on $C_\red$ necessarily imply that $C_A$ is acyclic and that $C_\red$ has no acyclic summands, because of the symmetry and rank properties of $C$.
\end{remark}

The sequence $a_1, \ldots, a_{2m}$ indicates the signed lengths of the arrows encountered as we trace $C_\red$ from the vertically distinguished element to the horizontally distinguished element along horizontal and vertical arrows. For $i$ odd, the arrow between $x_i$ and $x_{i-1}$ is horizontal, and $a_i$ is positive if it is outgoing from $x_i$, otherwise negative. For $i$ even, the arrow between $x_i$ and $x_{i-1}$ is vertical, and $a_i$ is positive if it is incoming to $x_i$, otherwise negative. The sign convention is chosen such that $L$-space knots will have all positive $a_i$. For example, the complex in Figure \ref{fig:T_3,4} is of type $(1, 2)$.

Note that if $C$ is a representative of the type $(a_1, \ldots, a_m)$, denoted $C \in (a_1, \ldots, a_m)$, then $C$ must have at least $2m+1$ basis elements. If $C$ has exactly $2m+1$ basis elements (i.e., $C \simeq C_\red$), then we say that $C$ is \emph{reduced}. Given $C_1 \in (a_1, \ldots, a_m)$ with basis $\{x_i\}$, $i = 0, \ldots, 2m$, we know that $C_1$ must be the reduced representative, and we assume the $x_i$ are labeled in the order described in Definition \ref{def:Ctype}. Moreover, if $C_2 \in (b_1, \ldots, b_n)$ with basis $\{y_j\}$, $j = 0, \ldots, 2n$, then $C_1 \otimes C_2$ naturally has (unsimplified) basis $\{x_iy_j\}$, where $x_iy_j = x_i \otimes y_j$.

The complex $C_1$ is a \emph{staircase complex} if all $a_i$ are positive, or if all $a_i$ are negative. By Lemma \ref{lem:typeepsilon} below, a staircase complex is $\varepsilon$-equivalent to a complex with no diagonal arrows, namely the reduced representative of $(a_1, \ldots, a_m)$ with exactly $2m+1$ generators. If each $a_i > 0$, i.e., $C_1$ is a staircase complex, we have
\begin{equation} \label{eqn:differential}
\partial x_i = \begin{cases}[1] x_{i-1} + x_{i+1} &i \text{ odd} \\ 0 &i \text{ even} \end{cases} \qquad \text{and} \qquad \partial x_i^* = \begin{cases}[1] x_{i+1}^* + x_{i-1}^* &i \text{ even},\, i \neq m \mp m \\ 0 &i \text{ odd} \\ x_{i\pm1}^* &i = m \mp m. \end{cases}
\end{equation}
Using $\fl(x_i)$ to denote the filtration level of $x_i$, we have $\fl(x_0) = (0, \tau(C_1))$. It is further clear that
\begin{equation} \label{eqn:fl}
\fl(x_i) = \begin{cases}[1] \fl(x_{i-1}) + (a_i, 0) &i \text{ odd} \\ \fl(x_{i-1}) - (0, a_i) &i \text{ even} \end{cases} \qquad \text{and} \qquad \fl(x_i^*) = -\fl(x_i).
\end{equation}
By definition, the operation of tensor product on chain complexes gives us
\begin{align}
\label{eqn:mixed differential} \partial(x_iy_j) &= \partial(x_i)y_j + x_i\partial(y_j) \\
\label{eqn:mixed fl} \fl(x_iy_j) &= \fl(x_i) + \fl(y_j).
\end{align}
Notice $\fl(x_0^*) = -(0, \tau(C_1)) = (0, \tau(C_1^*))$ and $\fl(x_0y_0) = (0, \tau(C_1)) + (0, \tau(C_2)) = (0, \tau(C_1 \otimes C_2))$.

The following lemma shows that the type of a complex determines an $\varepsilon$-equivalence class. That is, if $C$ is a representative of the type $T$ and the $\varepsilon$-equivalence class $[C]$, then $T \subseteq [C]$.

\begin{lemma} \label{lem:typeepsilon}
If two complexes $C_1$ and $C_2$ are both of type $(a_1, \ldots, a_m)$, then the complexes are $\varepsilon$-equivalent.
\end{lemma}

\begin{proof}
Without loss of generality, we may assume that $a_1$ is positive. Let $\{x_0, \ldots, x_{2m}\}$ be a basis as in Definition \ref{def:Ctype} for the reduced summand of $C_1$ and similarly $\{y_0, \ldots, y_{2m}\}$ a basis for the reduced summand of $C_2^*$, where $C_2^*$ denotes the dual of $C_2$; that is, if $\{z_0, \ldots, z_{2m}\}$ is a basis for the reduced summand of $C_2$ as in Definition \ref{def:Ctype}, then $y_i = z_i^*$. In particular, there is a horizontal arrow of length $a_1$ from $x_1$ to $x_0$ and a horizontal arrow of length $a_1$ from $y_0$ to $y_1$. More generally, if there is a vertical (respectively horizontal) arrow from $x_i$ to $x_j$, then there is a vertical (respectively horizontal) arrow from $y_j$ to $y_i$. (Note the order of the subscripts on $x$ and $y$.)

We need to show that $\varepsilon(C_1 \otimes C_2^*) = 0$. In light of \cite[Definition 3.1]{Homsmooth}, it is sufficient to show that there exists a class in $C_1 \otimes C_2^*$ that is non-zero in the homology of $C\{\max(i, j-\tau) = 0\}$ and $C\{\min(i, j-\tau) = 0\}$. See \cite[Section 3]{Homsmooth} for the definition of these complexes. We claim that
\[ u = \sum_{i=0}^{2m} x_iy_i \]
is such an element. Note that each term in the above sum is in the same filtration level.

We first consider the horizontal homology. We will show that $u$ is in the kernel of $\partial^\horz$. Suppose that $x_jy_i$ appears in $\partial^\horz u$. Then either there is a horizontal arrow from $x_i$ to $x_j$ or from $y_j$ to $y_i$, and $j = i\pm1$. But as noted above, there is a horizontal arrow from $x_i$ to $x_j$ exactly when there is a horizontal arrow from $y_j$ to $y_i$. In particular, $x_jy_i$ appears in $\partial^\horz u$ exactly twice: once from $\partial^\horz x_iy_i$ and once from $\partial^\horz x_jy_j$. In this situation, we also have that $\partial^\horz x_iy_j = x_iy_i + x_jy_j$.

We now show that $u$ is not in the image of $\partial^\horz$. From the last sentence of the preceding paragraph, we have that for each $i = 0, \ldots, m-1$, the sum $x_{2i}y_{2i} + x_{2i+1}y_{2i+1}$ is in the image of $\partial^\horz$. It follows that
\[ \smashoperator{\sum_{i=0}^{2m-1}} x_iy_i = u - x_{2m}y_{2m} \]
is in the image of $\partial^\horz$. Since $x_{2m}$ and $y_{2m}$ were each the distinguished horizontal element of their respective bases, the element $x_{2m}y_{2m}$ is not in the image of $\partial^\horz$, so neither is $u$.

Similarly, it follows that $u$ also generates the vertical homology (where the role of the element $x_{2m}y_{2m}$ is now played by $x_0y_0$). Moreover, similar arguments show that $u$ is non-zero in $H_*(C\{\max(i, j-\tau) = 0\})$ and $H_*(C\{\min(i, j-\tau) = 0\})$. We conclude that $\varepsilon(C_1 \otimes C_2^*) = 0$, implying that $C_1 \sim_\varepsilon C_2$.
\end{proof}

It follows that for $C$ of type $(a_1, \ldots, a_m)$, we may denote the element $[C]$ of $\cF_\text{alg}$ by
\[ [a_1, \ldots, a_m]. \]
Note that $-[a_1, \ldots, a_m] = [-a_1, \ldots, -a_m]$ and $[\;] = 0$. We will sometimes use nested iterators to write our sequences. For instance, given sequences $(a_{i,j})_{j=1}^{n_i}$ indexed by $j$ for $i = 1, \ldots, m$, we can form the sequence
\[ ((a_{i,j})_{j=1}^{n_i})_{i=1}^m = (a_{1,1}, a_{1,2}, \ldots a_{1,n_1}, a_{2,1}, a_{2,2}, \ldots a_{2,n_2}, \ldots, a_{m,1}, a_{m,2}, \ldots a_{m,n_m}). \]
We will also write $((a_j)_{j=1}^n)^m$ to denote the sequence given by $a_1, \ldots, a_n$ repeated $m$ times, i.e.,
\[ ((a_j)_{j=1}^n)^m = \overbrace{(a_1, a_2, \ldots, a_n, a_1, a_2, \ldots, a_n, \ldots, a_1, a_2, \ldots, a_n)}^m. \]
Note that not every sequence of integers $a_1, \ldots, a_m$ corresponds to an element $[a_1, \ldots, a_m]$ of $\cF_\text{alg}$. For example, $[1, -2]$ does not admit a chain complex representative, as there is no collection of diagonal arrows that makes $\partial^2 = 0$. See Figure \ref{fig:F_alg} for two examples.

\begin{figure}[htb]
\centering
\subfigure[]{%
\begin{tikzpicture}
	\draw [step=1, black!10!white, very thin] (-0.9, -0.9) grid (3.9, 3.9);
	\draw [gray, thin, <->] (-0.9, 0) -- (3.9, 0);
	\draw [gray, thin, <->] (0, -0.9) -- (0, 3.9);
	\filldraw (0, 1.1) circle (2pt) node[] (x0){};
	\filldraw (1.1, 1.1) circle (2pt) node[] (x1){};
	\filldraw (1.1, 0) circle (2pt) node[] (x2){};
	\filldraw (0, 0.9) circle (2pt) node[] (y0){};
	\filldraw (0.9, 0.9) circle (2pt) node[] (y1){};
	\filldraw (0.9, 0) circle (2pt) node[] (y2){};
	\filldraw (0, 0) circle (2pt) node[] (y3){};
	\draw [very thick, <-] (x0) -- (x1);
	\draw [very thick, <-] (x2) -- (x1);
	\draw [very thick, <-] (y0) -- (y1);
	\draw [very thick, <-] (y2) -- (y1);
	\draw [very thick, <-] (y3) -- (y2);
	\draw [very thick, <-] (y3) -- (y0);
	\node [left] at (0, 1.2) {$x_0$};
	\node [above] at (x1) {$x_1$};
	\node [below] at (1.25, 0) {$x_2$};
	\node [left] at (0, 0.8) {$y_0$};
	\node [] at (0.65, 0.65) {$y_1$};
	\node [below] at (0.75, 0) {$y_2$};
	\node [] at (-0.25, -0.25) {$y_3$};
\end{tikzpicture}}\hspace{40pt}%
\subfigure[]{%
\begin{tikzpicture}
	\draw [step=1, black!10!white, very thin] (-0.9, -0.9) grid (3.9, 3.9);
	\draw [gray, thin, <->] (-0.9, 0) -- (3.9, 0);
	\draw [gray, thin, <->] (0, -0.9) -- (0, 3.9);
	\filldraw (0, 2) circle (2pt) node[] (x0){};
	\filldraw (2.9, 2) circle (2pt) node[] (x1){};
	\filldraw (2.9, 3.1) circle (2pt) node[] (x2){};
	\filldraw (1, 3.1) circle (2pt) node[] (x3){};
	\filldraw (1, 1) circle (2pt) node[] (x4){};
	\filldraw (3.1, 1) circle (2pt) node[] (x5){};
	\filldraw (3.1, 2.9) circle (2pt) node[] (x6){};
	\filldraw (2, 2.9) circle (2pt) node[] (x7){};
	\filldraw (2, 0) circle (2pt) node[] (x8){};
	\draw [very thick, <-] (x0) -- (x1);
	\draw [very thick, <-] (x1) -- (x2);
	\draw [very thick, <-] (x3) -- (x2);
	\draw [very thick, <-] (x4) -- (x3);
	\draw [very thick, <-] (x4) -- (x5);
	\draw [very thick, <-] (x5) -- (x6);
	\draw [very thick, <-] (x7) -- (x6);
	\draw [very thick, <-] (x8) -- (x7);
	\draw [very thick, <-] (x0) -- (x3);
	\draw [very thick, <-] (x4) -- (x1);
	\draw [very thick, <-] (x4) -- (x7);
	\draw [very thick, <-] (x8) -- (x5);
	\node [left] at (x0) {$x_0$};
	\node [below] at (2.9, 1.95) {$x_1$};
	\node [above] at (x2) {$x_2$};
	\node [above] at (x3) {$x_3$};
	\node [below] at (x4) {$x_4$};
	\node [right] at (x5) {$x_5$};
	\node [right] at (x6) {$x_6$};
	\node [left] at (x7) {$x_7$};
	\node [below] at (x8) {$x_8$};
\end{tikzpicture}}
\caption{Left, a basis for $\CFKi(5_2)$, showing $[\CFKi(5_2)] = [1]$. Right, a basis for a reduced representative of $[3, -1, -2, 2]$.}
\label{fig:F_alg} \end{figure}
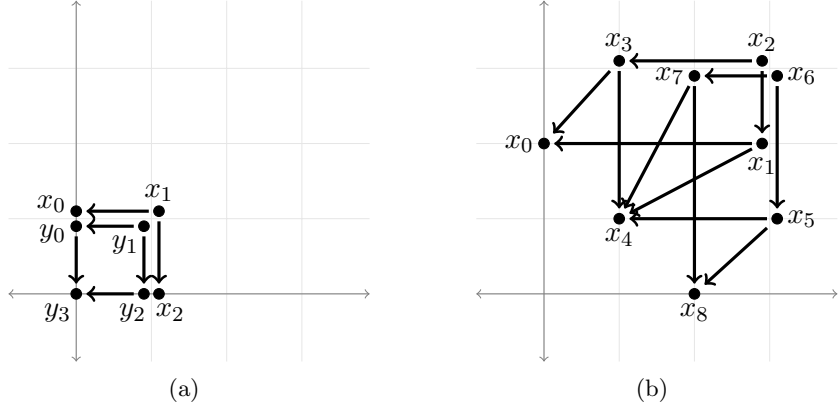

At times, it will be convenient to have some $a_i$ equal to zero. To this end, we make the formal identifications:
\begin{itemize}
	\item $[a_1, \ldots, a_m] = [a_1, \ldots, a_m, 0]$.
	\item $[a_1, \ldots, a_j, \ldots, a_m] = [a_1, \ldots, a_j-k, 0, k, \ldots, a_m] \quad \text{for } \min(0, a_j) \leq k \leq \max(0, a_j)$.
\end{itemize}
In particular, this allows $(\ref{eqn:staircase})$ to remain valid even when the sequence of $n_i$ defining $\Delta_K(t)$ is only non-strictly increasing. It follows that given $[a_1, \ldots, a_m] \in \cF_\text{alg}$ with $a_i = 0$, we may use $[\ldots, a_{i-1}, 0, a_{i+1}, \ldots] = [\ldots, a_{i-1}+a_{i+1}, \ldots]$ to remove the zero entry.

Finally, note that the concordance invariants $\tau(K)$ of \cite{OS4ball} and $\varepsilon(K)$, $a_1(K)$, and $a_2(K)$ of \cite{Homsmooth} are well-defined invariants of elements in $\cF_\text{alg}$. For $[C] = [a_1, \ldots, a_m] \in \cF_\text{alg}$, we have
\[ \tau(C) = \sum_{i=1}^m a_i \qquad \text{and} \qquad \varepsilon(C) = \begin{cases}[1] \text{sgn}(a_1) &\text{if } m > 0 \\ 0 &\text{otherwise}, \end{cases} \]
and if $\varepsilon(C) = 1$,
\[ a_1(C) = a_1 \qquad \text{and} \qquad a_2(C) = \begin{cases}[1] a_2 &\text{if } a_2 > 0 \\ \text{undefined} &\text{otherwise}. \end{cases} \]

\section{Tensor Products of Staircase Complexes} \label{sec:tensor}

For the calculations that follow in Sections \ref{sec:tensor} and \ref{sec:ordering}, we suppress the $U$-translates from this picture, which is always possible for the complexes under consideration here, given an appropriate choice of basis. That is, the complexes of interest are all of the form $C = C' \otimes_\F \F[U, U^{-1}]$, where $C'$ is a doubly filtered, finitely-generated complex over $\F$. Thus $C_1 \otimes_{\F[U, U^{-1}]} C_2 = (C_1' \otimes C_2') \otimes_\F \F[U, U^{-1}]$.

In this section, we prove two lemmas on the group operation of $\cF_\text{alg}$. Our approach is to take the tensor product of two complexes, both reduced with $\varepsilon = 1$, then vertically and horizontally simplify the basis of the product to determine the reduced representative of its $\varepsilon$-equivalence class.

\begin{lemma} \label{lem:box}
Let $a_i,\, b_j > 0$ for $i = 1, \dots, m$ and $j = 1, \dots, n$. If $m$ is even and $\max\{a_i \mid i \textup{ odd}\} \leq b_j \leq \min\{a_i \mid i \textup{ even}\}$, then
\[ [a_1, a_2, \ldots, a_m] + [b_1, b_2, \ldots, b_n] = [a_1, a_2, \ldots, a_m, b_1, b_2, \ldots, b_n]. \]
\end{lemma}

\begin{proof}
For $C_1 \in (a_1, \ldots, a_m)$ with basis $\{x_i\}$, $i = 0, \ldots, 2m$, and $C_2 \in (b_1, \ldots, b_n)$ with basis $\{y_j\}$, $j = 0, \ldots, 2n$, we prove that $C_1 \otimes C_2$ is of type $T = (a_1, \ldots, a_m, b_1, \ldots, b_n)$. Define the sets
\begin{align*}
S &= \{x_iy_0 \mid i < m\} \cup \{x_my_j \mid 0 \leq j \leq 2n\} \cup \{x_iy_{2n} \mid i > m\} \\
B_{i,j} &= \begin{cases}[1] \{x_iy_j, x_{i-1}y_j, x_iy_{j+1}, x_{i-1}y_{j+1}\} &i, j \text{ odd},\, i < m \\ \{x_iy_j, x_{i+1}y_j, x_iy_{j-1}, x_{i+1}y_{j-1}\} &i, j \text{ odd},\, i > m. \end{cases}
\end{align*}
Supposing $m$ is even, $\bigcup_{i<m,j} B_{i,j} = \{x_iy_j \mid i < m,\, j > 0\}$ and $\bigcup_{i>m,j} B_{i,j} = \{x_iy_j \mid i > m,\, j < 2n\}$. Therefore, $\{x_iy_j\} = S \cup \{\bigcup_{i,j} B_{i,j}\}$. Note that $S$ is the elements $\{x_iy_0 \mid i < m\} \cup \{x_my_j \mid j \leq n\}$ along with their reflection about the diagonal. Given the differentials on these elements,
\[ \partial(x_iy_0) = \begin{cases}[1] x_{i-1}y_0 + x_{i+1}y_0 &i \text{ odd} \\ 0 &i \text{ even} \end{cases} \qquad \text{and} \qquad \partial(x_my_j) = \begin{cases}[1] x_my_{j-1} + x_my_{j+1} &j \text{ odd} \\ 0 &j \text{ even}, \end{cases} \]
and their filtration levels,
\[ \fl(x_iy_0) = \begin{cases}[1] \fl(x_{i-1}y_0) + (a_i, 0) &i \text{ odd} \\ \fl(x_{i-1}y_0) - (0, a_i) &i \text{ even} \end{cases} \qquad \text{and} \qquad \fl(x_my_j) = \begin{cases}[1] \fl(x_my_{j-1}) + (b_j, 0) &j \text{ odd} \\ \fl(x_my_{j-1}) - (0, b_j) &j \text{ even}, \end{cases} \]
it is clear that $S$ spans a subcomplex isomorphic to the reduced representative of type $T$. We show that each of the remaining $mn$ sets of elements $B_{i,j}$ forms an acyclic summand under a filtered change of basis. These summands look like boxes---see Figure \ref{fig:boxes} for an example. Since $B_{i,j} \to B_{2m-i,2n-j}$ when $x_iy_j \to x_{2m-i}y_{2n-j}$, the proofs for $i < m$ and $i > m$ are the same under the transformation $x_iy_j \to x_{2m-i}y_{2n-j}$, $a_i \to a_{2m+1-i}$, and $b_j \to b_{2n+1-j}$ by the diagonal symmetry of $C_1 \otimes C_2$. It therefore suffices to redefine the basis for $i = 0, \ldots, m-1$.

\begin{figure}[htb]
\centering
\subfigure[]{%
	\begin{tikzpicture}
	\draw [step=1, black!10!white, very thin] (0, 0) grid (6.6, 6.6);
	\draw [gray, thin, ->] (0, 0) -- (6.6, 0);
	\draw [gray, thin, ->] (0, 0) -- (0, 6.6);
	\filldraw (0, 6) circle (2pt) node[] (x0y0){};
	\filldraw (1, 6) circle (2pt) node[] (x1y0){};
	\filldraw (1, 2.9) circle (2pt) node[] (x2y0){};
	\filldraw (4, 2.9) circle (2pt) node[] (x3y0){};
	\filldraw (4, 2) circle (2pt) node[] (x4y0){};
	\filldraw (2, 6) circle (2pt) node[] (x0y1){};
	\filldraw (3.1, 6) circle (2pt) node[] (x1y1){};
	\filldraw (3.1, 3.1) circle (2pt) node[] (x2y1){};
	\filldraw (6, 3.1) circle (2pt) node[] (x3y1){};
	\filldraw (6, 2) circle (2pt) node[] (x4y1){};
	\filldraw (2, 4) circle (2pt) node[] (x0y2){};
	\filldraw (2.9, 4) circle (2pt) node[] (x1y2){};
	\filldraw (2.9, 1) circle (2pt) node[] (x2y2){};
	\filldraw (6, 1) circle (2pt) node[] (x3y2){};
	\filldraw (6, 0) circle (2pt) node[] (x4y2){};
	\draw [very thick, <-] (x0y0) -- (x1y0);
	\draw [very thick, <-] (x2y0) -- (x1y0);
	\draw [very thick, <-] (x2y0) -- (x3y0);
	\draw [very thick, <-] (x4y0) -- (x3y0);
	\draw [very thick, <-] (x0y1) -- (x1y1);
	\draw [very thick, <-] (x2y1) -- (x1y1);
	\draw [very thick, <-] (x2y1) -- (x3y1);
	\draw [very thick, <-] (x4y1) -- (x3y1);
	\draw [very thick, <-] (x0y2) -- (x1y2);
	\draw [very thick, <-] (x2y2) -- (x1y2);
	\draw [very thick, <-] (x2y2) -- (x3y2);
	\draw [very thick, <-] (x4y2) -- (x3y2);
	\draw [very thick, <-] (x2y0) -- (x1y0);
	\draw [very thick, <-] (x0y0) .. controls (0.4, 5.6) and (1.6, 5.6) .. (x0y1);
	\draw [very thick, <-] (x0y2) .. controls (1.6, 4.4) and (1.6, 5.6) .. (x0y1);
	\draw [very thick, <-] (x1y0) .. controls (1.4, 6.4) and (2.6, 6.4) .. (x1y1);
	\draw [very thick, <-] (x1y2) .. controls (2.5, 4.4) and (2.7, 5.6) .. (x1y1);
	\draw [very thick, <-] (x2y0) .. controls (1.4, 3.3) and (2.6, 3.5) .. (x2y1);
	\draw [very thick, <-] (x2y2) .. controls (3.3, 1.4) and (3.5, 2.6) .. (x2y1);
	\draw [very thick, <-] (x3y0) .. controls (4.4, 2.5) and (5.6, 2.7) .. (x3y1);
	\draw [very thick, <-] (x3y2) .. controls (6.4, 1.4) and (6.4, 2.6) .. (x3y1);
	\draw [very thick, <-] (x4y0) .. controls (4.4, 1.6) and (5.6, 1.6) .. (x4y1);
	\draw [very thick, <-] (x4y2) .. controls (5.6, 0.4) and (5.6, 1.6) .. (x4y1);
	\node [left] at (x0y0) {$x_0y_0$};
	\node [above] at (0.6, 6) {$x_1y_0$};
	\node [left] at (x2y0) {$x_2y_0$};
	\node [right] at (x3y0) {$x_3y_0$};
	\node [right] at (x4y0) {$x_4y_0$};
	\node [below] at (2.3, 5.95) {$x_0y_1$};
	\node [right] at (x1y1) {$x_1y_1$};
	\node [] at (3.66, 3.33) {$x_2y_1$};
	\node [above] at (x3y1) {$x_3y_1$};
	\node [left] at (x4y1) {$x_4y_1$};
	\node [left] at (2, 3.95) {$x_0y_2$};
	\node [right] at (3.05, 4) {$x_1y_2$};
	\node [below] at (x2y2) {$x_2y_2$};
	\node [right] at (6, 0.9) {$x_3y_2$};
	\node [right] at (6.05, 0.25) {$x_4y_2$};
\end{tikzpicture}}%
\subfigure[]{%
	\begin{tikzpicture}
	\draw [step=1, black!10!white, very thin] (0, 0) grid (6.6, 6.6);
	\draw [gray, thin, ->] (0, 0) -- (6.6, 0);
	\draw [gray, thin, ->] (0, 0) -- (0, 6.6);
	\filldraw (0, 6) circle (2pt) node[] (x0y0){};
	\filldraw (1, 6) circle (2pt) node[] (x1y0){};
	\filldraw (1, 3) circle (2pt) node[] (x2y0){};
	\filldraw (3, 3) circle (2pt) node[] (x2y1){};
	\filldraw (3, 1) circle (2pt) node[] (x2y2){};
	\filldraw (6, 1) circle (2pt) node[] (x3y2){};
	\filldraw (6, 0) circle (2pt) node[] (x4y2){};
	\filldraw (3, 6) circle (2pt) node[] (x1y1){};
	\filldraw (2, 6) circle (2pt) node[] (x0y1){};
	\filldraw (3, 4) circle (2pt) node[] (x1y2){};
	\filldraw (2, 4) circle (2pt) node[] (x0y2){};
	\filldraw (6, 3) circle (2pt) node[] (x3y1){};
	\filldraw (6, 2) circle (2pt) node[] (x4y1){};
	\filldraw (4, 3) circle (2pt) node[] (x3y0){};
	\filldraw (4, 2) circle (2pt) node[] (x4y0){};
	\draw [very thick, <-] (x0y0) -- (x1y0);
	\draw [very thick, <-] (x2y0) -- (x1y0);
	\draw [very thick, <-] (x2y0) -- (x2y1);
	\draw [very thick, <-] (x2y2) -- (x2y1);
	\draw [very thick, <-] (x2y2) -- (x3y2);
	\draw [very thick, <-] (x4y2) -- (x3y2);
	\draw [very thick, <-] (x0y1) -- (x1y1);
	\draw [very thick, <-] (x1y2) -- (x1y1);
	\draw [very thick, <-] (x0y2) -- (x0y1);
	\draw [very thick, <-] (x0y2) -- (x1y2);
	\draw [very thick, <-] (x4y1) -- (x3y1);
	\draw [very thick, <-] (x3y0) -- (x3y1);
	\draw [very thick, <-] (x4y0) -- (x4y1);
	\draw [very thick, <-] (x4y0) -- (x3y0);
	\node [left] at (x0y0) {$x_0y_0$};
	\node [above] at (0.6, 6) {$x_1y_0$};
	\node [left] at (x2y0) {$x_2y_0$};
	\node [below] at (2.55, 3) {$x_2y_1$};
	\node [below] at (x2y2) {$x_2y_2$};
	\node [below] at (5.55, 1) {$x_3y_2$};
	\node [right] at (6.05, 0.25) {$x_4y_2$};
	\node [right] at (x1y1) {$x_1y_1$};
	\node [above] at (2.15, 6) {$x_0y_1{+}x_1y_0$};
	\node [right] at (x1y2) {$x_1y_2{+}x_2y_1$};
	\node [below] at (2.15, 4) {$x_0y_2{+}x_2y_0$};
	\node [above] at (x3y1) {$x_3y_1$};
	\node [below] at (6.05, 2) {$x_4y_1{+}x_3y_2$};
	\node [above] at (x3y0) {$x_3y_0{+}x_2y_1$};
	\node [below] at (x4y0) {$x_4y_0{+}x_2y_2$};
\end{tikzpicture}}
\caption{$C_1 \otimes C_2$ with, left, basis $\{x_iy_j\}$ and, right, simplified basis $S \cup B'_{1,1} \cup B'_{3,1}$, where $C_1 \in (1, 3)$ has basis $\{x_0, x_1, x_2, x_3, x_4\}$ and $C_2 \in (2)$ has basis $\{y_0, y_1, y_2\}$.}
\label{fig:boxes} \end{figure}
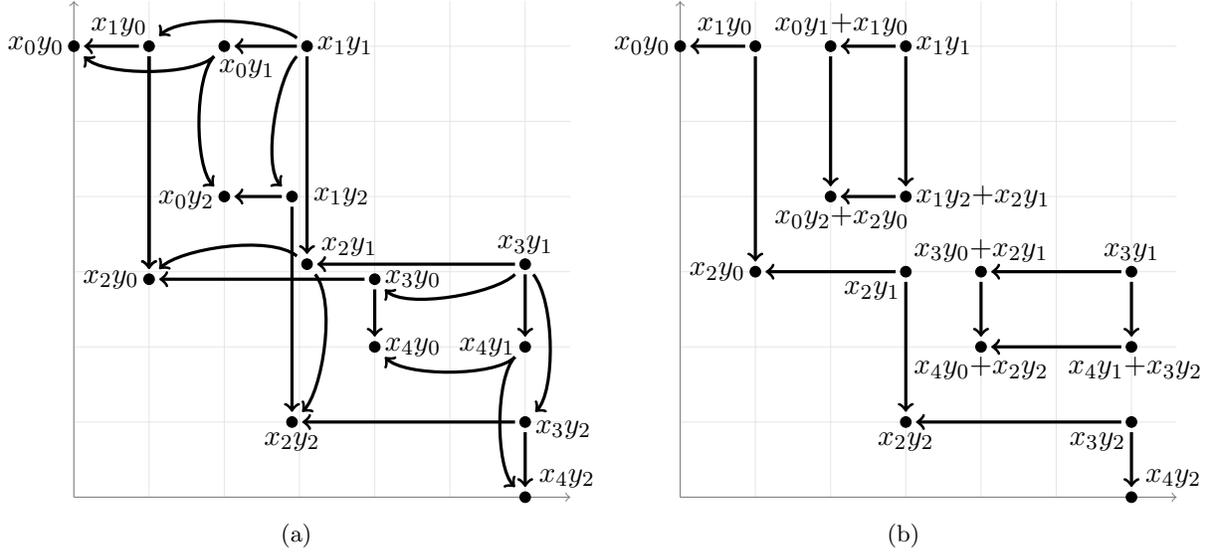

For all odd $i < m$ and all odd $j$, produce the new basis
\begin{equation} \label{eqn:cob xy}
x_{i-1}y_j \to x_{i-1}y_j + x_iy_{j-1}, \qquad x_iy_{j+1} \to x_iy_{j+1} + x_{i+1}y_j, \qquad x_{i-1}y_{j+1} \to x_{i-1}y_{j+1} + x_{i+1}y_{j-1}.
\end{equation}
It follows that $B_{i,j} \to B'_{i,j}$, where $B'_{i,j} = \{x_iy_j, x_{i-1}y_j + x_iy_{j-1}, x_iy_{j+1} + x_{i+1}y_j, x_{i-1}y_{j+1} + x_{i+1}y_{j-1}\}$.

First, we check that our chosen basis respects the filtration. By (\ref{eqn:fl}, \ref{eqn:mixed fl}) with $i$ and $j$ odd,
\begin{equation} \label{eqn:fl xy}
\begin{aligned}
\fl(x_{i-1}y_j) &= (\fl(x_i) - (a_i, 0)) + (\fl(y_{j-1}) + (b_j, 0)) \\
&= \fl(x_iy_{j-1}) + (b_j-a_i, 0) \\
\fl(x_iy_{j+1}) &= (\fl(x_{i+1}) + (0, a_{i+1})) + (\fl(y_j) - (0, b_{j+1})) \\
&= \fl(x_{i+1}y_j) + (0, a_{i+1}-b_{j+1}) \\
\fl(x_{i-1}y_{j+1}) &= (\fl(x_{i+1}) + (-a_i, a_{i+1})) + (\fl(y_{j-1}) + (b_j, -b_{j+1})) \\
&= \fl(x_{i+1}y_{j-1}) + (b_j-a_i, a_{i+1}-b_{j+1}).
\end{aligned}
\end{equation}
We require $\fl(x_{i-1}y_j) \geq \fl(x_iy_{j-1})$, $\fl(x_iy_{j+1}) \geq \fl(x_{i+1}y_j)$, and $\fl(x_{i-1}y_{j+1}) \geq \fl(x_{i+1}y_{j-1})$. By (\ref{eqn:fl xy}), this is equivalent to $a_i \leq b_j$ and $b_{j+1} \leq a_{i+1}$ for $i$ and $j$ odd, which in turn is equivalent to $a_i \leq b_j \leq a_{i+1}$ for $i$ odd and $j = 1, \dots, n$, which is true by hypothesis.

Second, we check that the differential on $B'_{i,j}$ gives an acyclic summand. Using (\ref{eqn:differential}, \ref{eqn:mixed differential}) with $i$ and $j$ odd and $\Z/2\Z$ coefficients,
\begin{equation}
\begin{aligned}
\partial(x_iy_j) &= x_{i-1}y_j + x_{i+1}y_j + x_iy_{j-1} + x_iy_{j+1} \\
&= (x_{i-1}y_j + x_iy_{j-1}) + (x_iy_{j+1} + x_{i+1}y_j) \\
\partial(x_{i-1}y_j + x_iy_{j-1}) &= (x_{i-1}y_{j-1} + x_{i-1}y_{j+1}) + (x_{i-1}y_{j-1} + x_{i+1}y_{j-1}) \\
&= x_{i-1}y_{j+1} + x_{i+1}y_{j-1} \\
\partial(x_iy_{j+1} + x_{i+1}y_j) &= (x_{i-1}y_{j+1} + x_{i+1}y_{j+1}) + (x_{i+1}y_{j-1} + x_{i+1}y_{j+1}) \\
&= x_{i-1}y_{j+1} + x_{i+1}y_{j-1} \\
\partial(x_{i-1}y_{j+1} + x_{i+1}y_{j-1}) &= 0.
\end{aligned}
\end{equation}
That is, both $(x_{i-1}y_j + x_iy_{j-1})$ and $(x_iy_{j+1} + x_{i+1}y_j)$ have one incoming arrow from $x_iy_j$ and one outgoing arrow to $(x_{i-1}y_{j+1} + x_{i+1}y_{j-1})$, joining the former to the latter. As there are no outgoing arrows from $S$ to $\{x_iy_j\} \backslash S$ or $B'_{i,j}$ to $\{x_iy_j\} \backslash B_{i,j}$, neither are there incoming arrows to $B'_{i,j}$. We conclude that each $B'_{i,j}$ is a basis for an acyclic subcomplex that splits off as a direct summand.
\end{proof}

Two examples that follow directly from inductive application of Lemma \ref{lem:box} are $n[a] = [(a)^n]$ (using the relation $2[a] = [a, a]$) and $n[a_1, a_2] = [(a_1, a_2)^n]$ if $a_1 \leq a_2$. As a point of constrast, we note without proof that $n[a_1, a_2] = [(c_i)_{i=1}^{2n}]$ with $(c_i)_{i=1}^{4n} = (a_1, a_2, a_2, a_1)^n$ if $a_1 \geq a_2$.

\begin{lemma} \label{lem:polygon}
Let $a > 0$ and $c,\, d_\ell,\, p,\, q_\ell \geq 0$ for $\ell = 1, \ldots, r$. If $\min\{q_\ell\} \geq p$ and $\max\{d_\ell\} \leq c$, then
\[ [(1, a)^p\!, 1, a+c] + [((1, a)^{q_\ell}\!, 1, a+d_\ell)_{\ell=1}^r] = [(1, a)^p\!, 1, a+c, ((1, a)^{q_\ell}\!, 1, a+d_\ell)_{\ell=1}^r]. \]
\end{lemma}

\begin{proof}
Set $m = 2p+2$ and $n_\ell = 2\sum_{k=1}^\ell q_k + 2\ell$, where $q_\ell = q_{2r+1-\ell}$. For $C_1 \in ((1, a)^p\!, 1, a+c)$ with basis $\{x_i\}$, $i = 0, \ldots, 2m$, and $C_2 \in ((1, a)^{q_\ell}\!, 1, a+d_\ell)_{\ell=1}^r$ with basis $\{y_j\}$, $j = 0, \ldots, 2n_r$, we prove that $C_1 \otimes C_2$ is of type $T = ((1, a)^p\!, 1, a+c, ((1, a)^{q_\ell}\!, 1, a+d_\ell)_{\ell=1}^r)$. Define the sets
\begin{align*}
S &= \{x_iy_0 \mid i < m\} \cup \{x_my_j \mid 0 \leq j \leq 2n_r\} \cup \{x_iy_{2n_r} \mid i > m\} \\
B_{i,j} &= \begin{cases}[1] \{x_iy_j, x_{i-1}y_j, x_iy_{j+1}, x_{i-1}y_{j+1}\} &i, j \text{ odd},\, i < m,\, j \notin \{n_\ell-1 \mid \ell \leq r\} \\ \{x_iy_j, x_{i+1}y_j, x_iy_{j-1}, x_{i+1}y_{j-1}\} &i, j \text{ odd},\, i > m,\, j \notin \{n_\ell+1 \mid \ell \geq r\} \end{cases} \\
R_j &= \begin{cases}[1] \{x_iy_j, x_iy_{j+1} \mid i < m\} &j = n_\ell-1,\, \ell \leq r \\ \{x_iy_j, x_iy_{j-1} \mid i > m\} &j = n_\ell+1,\, \ell \geq r. \end{cases}
\end{align*}
Reasoning as in the preceding proof, $\{x_iy_j\} = S \cup \{\bigcup_{i,j} B_{i,j}\} \cup \{\bigcup_j R_j\}$, and $S$ spans a subcomplex isomorphic to the reduced representative of type $T$. We show that each of the $m(n_r-2r)$ sets of elements $B_{i,j}$ and $2r$ sets of elements $R_j$ forms an acyclic summand under a filtered change of basis. The former summands look like boxes and the latter look like rectilinear polygons with $4p+4$ sides---see Figure \ref{fig:polygons} for an example. Since $B_{i,j} \to B_{2m-i,2n_r-j}$ and $R_j \to R_{2n_r-j}$ when $x_iy_j \to x_{2m-i}y_{2n_r-j}$, it suffices to redefine the basis for $i = 0, \ldots, m-1$.

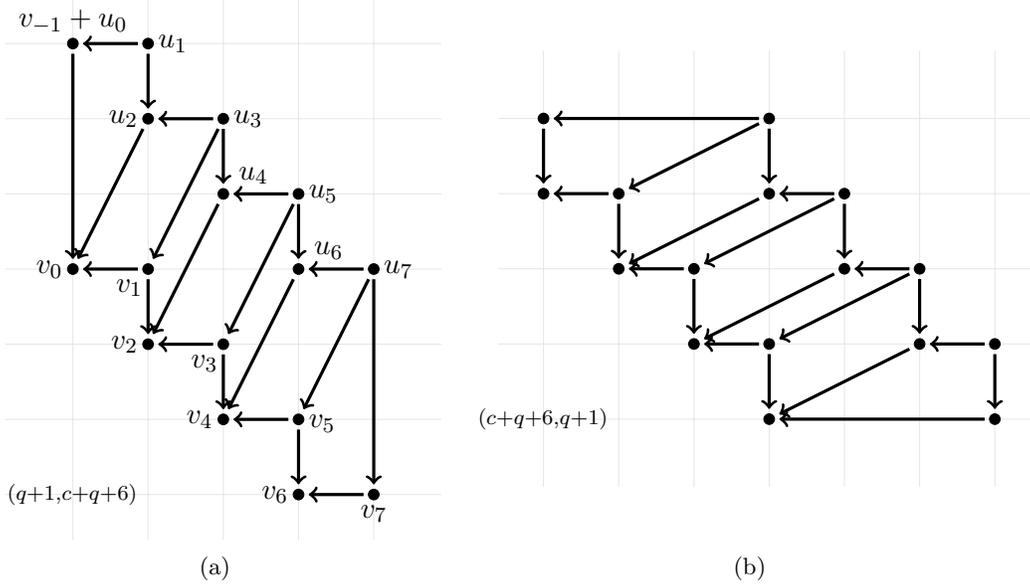
\begin{figure}[htb]
\centering
\subfigure[]{%
\begin{tikzpicture}
	\draw [step=1, black!10!white, very thin] (-0.9, -0.6) grid (4.9, 6.6);
	\filldraw (0, 6) circle (2pt) node[] (u0){};
	\filldraw (1, 6) circle (2pt) node[] (u1){};
	\filldraw (1, 5) circle (2pt) node[] (u2){};
	\filldraw (2, 5) circle (2pt) node[] (u3){};
	\filldraw (2, 4) circle (2pt) node[] (u4){};
	\filldraw (3, 4) circle (2pt) node[] (u5){};
	\filldraw (3, 3) circle (2pt) node[] (u6){};
	\filldraw (4, 3) circle (2pt) node[] (u7){};
	\filldraw (4, 0) circle (2pt) node[] (v7){};
	\filldraw (3, 0) circle (2pt) node[] (v6){};
	\filldraw (3, 1) circle (2pt) node[] (v5){};
	\filldraw (2, 1) circle (2pt) node[] (v4){};
	\filldraw (2, 2) circle (2pt) node[] (v3){};
	\filldraw (1, 2) circle (2pt) node[] (v2){};
	\filldraw (1, 3) circle (2pt) node[] (v1){};
	\filldraw (0, 3) circle (2pt) node[] (v0){};
	\draw [very thick, <-] (u0) -- (u1);
	\draw [very thick, <-] (u2) -- (u1);
	\draw [very thick, <-] (u2) -- (u3);
	\draw [very thick, <-] (u4) -- (u3);
	\draw [very thick, <-] (u4) -- (u5);
	\draw [very thick, <-] (u6) -- (u5);
	\draw [very thick, <-] (u6) -- (u7);
	\draw [very thick, <-] (v7) -- (u7);
	\draw [very thick, <-] (v6) -- (v7);
	\draw [very thick, <-] (v6) -- (v5);
	\draw [very thick, <-] (v4) -- (v5);
	\draw [very thick, <-] (v4) -- (v3);
	\draw [very thick, <-] (v2) -- (v3);
	\draw [very thick, <-] (v2) -- (v1);
	\draw [very thick, <-] (v0) -- (v1);
	\draw [very thick, <-] (v0) -- (u0);
	\draw [very thick, <-] (v0) -- (u2);
	\draw [very thick, <-] (v1) -- (u3);
	\draw [very thick, <-] (v2) -- (u4);
	\draw [very thick, <-] (v3) -- (u5);
	\draw [very thick, <-] (v4) -- (u6);
	\draw [very thick, <-] (v5) -- (u7);
	\node [above] at (u0) {$v_{-1}+u_0$};
	\node [right] at (u1) {$u_1$};
	\node [left] at (u2) {$u_2$};
	\node [right] at (u3) {$u_3$};
	\node [] at (2.4, 4.25) {$u_4$};
	\node [right] at (u5) {$u_5$};
	\node [] at (3.4, 3.25) {$u_6$};
	\node [right] at (u7) {$u_7$};
	\node [below] at (v7) {$v_7$};
	\node [left] at (v6) {$v_6$};
	\node [right] at (3, 0.95) {$v_5$};
	\node [left] at (v4) {$v_4$};
	\node [] at (1.75, 1.75) {$v_3$};
	\node [left] at (v2) {$v_2$};
	\node [] at (0.75, 2.75) {$v_1$};
	\node [left] at (v0) {$v_0$};
	\node [] at (0, 0) {$_{(q+1, c+q+6)}$};
\end{tikzpicture}}\hspace{10pt}%
\subfigure[]{%
\begin{tikzpicture}
	\draw [step=1, black!10!white, very thin] (-0.6, -0.9) grid (6.6, 4.9);
	\filldraw (6, 0) circle (2pt) node[] (x0){};
	\filldraw (6, 1) circle (2pt) node[] (x1){};
	\filldraw (5, 1) circle (2pt) node[] (x2){};
	\filldraw (5, 2) circle (2pt) node[] (x3){};
	\filldraw (4, 2) circle (2pt) node[] (x4){};
	\filldraw (4, 3) circle (2pt) node[] (x5){};
	\filldraw (3, 3) circle (2pt) node[] (x6){};
	\filldraw (3, 4) circle (2pt) node[] (x7){};
	\filldraw (0, 4) circle (2pt) node[] (x8){};
	\filldraw (0, 3) circle (2pt) node[] (x9){};
	\filldraw (1, 3) circle (2pt) node[] (x10){};
	\filldraw (1, 2) circle (2pt) node[] (x11){};
	\filldraw (2, 2) circle (2pt) node[] (x12){};
	\filldraw (2, 1) circle (2pt) node[] (x13){};
	\filldraw (3, 1) circle (2pt) node[] (x14){};
	\filldraw (3, 0) circle (2pt) node[] (x15){};
	\draw [very thick, <-] (x0) -- (x1);
	\draw [very thick, <-] (x2) -- (x1);
	\draw [very thick, <-] (x2) -- (x3);
	\draw [very thick, <-] (x4) -- (x3);
	\draw [very thick, <-] (x4) -- (x5);
	\draw [very thick, <-] (x6) -- (x5);
	\draw [very thick, <-] (x6) -- (x7);
	\draw [very thick, <-] (x8) -- (x7);
	\draw [very thick, <-] (x9) -- (x8);
	\draw [very thick, <-] (x9) -- (x10);
	\draw [very thick, <-] (x11) -- (x10);
	\draw [very thick, <-] (x11) -- (x12);
	\draw [very thick, <-] (x13) -- (x12);
	\draw [very thick, <-] (x13) -- (x14);
	\draw [very thick, <-] (x15) -- (x14);
	\draw [very thick, <-] (x15) -- (x0);
	\draw [very thick, <-] (x15) -- (x2);
	\draw [very thick, <-] (x14) -- (x3);
	\draw [very thick, <-] (x13) -- (x4);
	\draw [very thick, <-] (x12) -- (x5);
	\draw [very thick, <-] (x11) -- (x6);
	\draw [very thick, <-] (x10) -- (x7);
	\draw (0, -1.6);
	\node [] at (0, 0) {$_{(c+q+6, q+1)}$};
\end{tikzpicture}}
\caption{Acyclic summands of $C_1 \otimes C_2$ with bases, left, $R'_{2q+1}$ and, right, $R'_{2q+3}$, where $C_1 \in ((1, 1)^3\!, 1, 1+c)$ and $C_2 \in ((1, 1)^q\!, 1, 3)$ are reduced, $c \geq 2$, and $q \geq 3$.}
\label{fig:polygons} \end{figure}

Let $(a_i)_{i=1}^m = ((1, a)^p\!, 1, a+c)$ and $(b_j)_{j=1}^{n_r} = ((1, a)^{q_\ell}\!, 1, a+d_\ell)_{\ell=1}^r$. For $i$ and $j$ odd, $i < m$, and $j \notin \{n_\ell-1 \mid \ell \leq r\}$, we have $\max\{a_i\} = \min\{b_j\} = 1$ and $\max\{b_{j+1}\} = \min\{a_{i+1}\} = a$. Hence by (\ref{eqn:fl xy}), the basis change (\ref{eqn:cob xy}) is valid for each $B_{i,j}$. Recall that $B'_{i,j}$ is a basis for a box subcomplex.

Define the basis element sums
\begin{equation} \label{eqn:cob uv}
u_{i,j} = \smashoperator{\sum_{k=-1}^{m-i}} x_{i+k}y_{j-k} \qquad \text{and} \qquad v_{i,j} = x_iy_j + x_my_{j+i-m}.
\end{equation}
Supposing $q_\ell \geq p$ for each $\ell$, we may produce the new basis
\begin{equation}
x_iy_j \to u_{i,j},\  i \neq 0, \qquad x_iy_{j+1} \to v_{i,j+1}, \qquad x_0y_j \to u_{0,j} + v_{-1,j+1}
\end{equation}
for all $i < m$ and all $j \in \{n_\ell-1 \mid \ell \leq r\}$. Note here that $u_{0,j} + v_{-1,j+1} = \sum_{k=0}^{m-1} x_ky_{j-k}$. It follows that $R_j \to R'_j$, where $R'_j = \{u_{i,j}, v_{i,j+1}, u_{0,j} + v_{-1,j+1} \mid i < m\} \backslash \{u_{0,j}\}$.

First, we check that our basis respects the filtration. Let $0 < k < m-i$. By (\ref{eqn:fl}, \ref{eqn:mixed fl}) with $i$ even and $j = n_\ell-1$,
\begin{align*}
\fl(x_iy_j) &= (\fl(x_{i-1}) - (0, a)) + (\fl(y_{j+1}) + (0, a+d_\ell)) \\
&= \fl(x_{i-1}y_{j+1}) + (0, d_\ell),\  i \neq 0 \\
\fl(x_{i+k}y_{j-k}) &= \begin{cases}[1] (\fl(x_{i+k-1}) + (1, 0)) + (\fl(y_{j-k+1}) - (1, 0)) &k \text{ odd} \\ (\fl(x_{i+k-1}) - (0, a)) + (\fl(y_{j-k+1}) + (0, a)) &k \text{ even} \end{cases} \\
&= \fl(x_{i+k-1}y_{j-k+1}) \\
\fl(x_my_{j+i-m}) &= (\fl(x_{m-1}) - (0, a+c)) + (\fl(y_{j+i-m+1}) + (0, a)) \\
&= \fl(x_{m-1}y_{j+i-m+1}) - (0, c),\  i \neq 0
\end{align*}
and $\fl(x_my_j) = \fl(x_{m-1}y_{j+1}) - (0, c-d_\ell)$. By (\ref{eqn:fl}, \ref{eqn:mixed fl}) with $i$ odd and $j = n_\ell-1$,
\begin{align*}
\fl(x_iy_j) &= (\fl(x_{i-1}) + (1, 0)) + (\fl(y_{j+1}) + (0, a+d_\ell)) \\
&= \fl(x_{i-1}y_{j+1}) + (1, a+d_\ell) \\
\fl(x_{i+k}y_{j-k}) &= \begin{cases}[1] (\fl(x_{i+k-1}) - (0, a)) + (\fl(y_{j-k+1}) - (1, 0)) &k \text{ odd} \\ (\fl(x_{i+k-1}) + (1, 0)) + (\fl(y_{j-k+1}) + (0, a)) &k \text{ even} \end{cases} \\
&= \fl(x_{i+k-1}y_{j-k+1}) + (-1)^k(1, a) \\
\fl(x_my_{j+i-m}) &= (\fl(x_{m-1}) - (0, a+c)) + (\fl(y_{j+i-m+1}) - (1, 0)) \\
&= \fl(x_{m-1}y_{j+i-m+1}) - (1, a+c).
\end{align*}
Letting $e_i = i \bmod 2$, it follows that
\begin{equation} \label{eqn:fl uv}
\begin{aligned}
\fl(x_iy_j) &= \begin{cases}[1] \fl(x_{i+k}y_{j-k}) + e_{ik}(1, a) \\ \fl(x_{i-1}y_{j+1}) + (0, d_\ell) + e_i(1, a) &i \neq 0 \\ \fl(x_my_{j+i-m}) + (0, c) + e_i(1, a) &i \neq 0 \end{cases} \\
\fl(x_iy_{j+1}) &= \fl(x_my_{j+1+i-m}) + (0, c-d_\ell).
\end{aligned}
\end{equation}
We require $\fl(x_iy_{j+1}) \geq \fl(x_my_{j+1+i-m})$ and $\fl(x_iy_j) \geq \fl(x_{i+k}y_{j-k})$ for each $k$ summed over in $u_{i,j}$ ($i \neq 0$) and $v_{-1,j+1} + u_{0,j}$. By (\ref{eqn:fl uv}), this is equivalent to $0 \leq d_\ell \leq c$, which is true by hypothesis.

Second, we check that the differential on $R'_j$ gives an acyclic summand. Using (\ref{eqn:differential}, \ref{eqn:mixed differential}) with $i$ even and $j$ odd,
\begin{align*}
\partial u_{i,j} &= \partial \smashoperator{\sum_{k=-1}^{m-i}} x_{i+k}y_{j-k} = \partial \smashoperator{\sum_{\substack{k=-1 \\ k\text{ odd}}}^{m-i-1}} (x_{i+k}y_{j-k} + x_{i+k+1}y_{j-k-1}) \\
&= \smashoperator{\sum_{\substack{k=-1 \\ k\text{ odd}}}^{m-i-1}} ((x_{i+k-1}y_{j-k} + x_{i+k+1}y_{j-k}) + (x_{i+k+1}y_{j-k-2} + x_{i+k+1}y_{j-k})) \\
&= \smashoperator{\sum_{\substack{k=-1 \\ k\text{ odd}}}^{m-i-1}} (x_{i+k-1}y_{j-k} + x_{i+k+1}y_{j-k-2}) + \smashoperator{\sum_{\substack{k=-1 \\ k\text{ odd}}}^{m-i-1}} 2x_{i+k+1}y_{j-k} \\
&= (x_{i-2}y_{j+1} + x_my_{j+i-m-1}) + \smashoperator{\sum_{\substack{k=1 \\ k\text{ odd}}}^{m-i-1}} 2x_{i+k-1}y_{j-k} + \smashoperator{\sum_{\substack{k=-1 \\ k\text{ odd}}}^{m-i-1}} 2x_{i+k+1}y_{j-k} \\
&= v_{i-2,j+1} \\
\partial v_{i,j+1} &= \partial(x_iy_{j+1} + x_my_{j+1+i-m}) \\
&= 0.
\end{align*}
Note that we first broke up the sum $u_{i,j}$ into terms with strictly odd or even subscripts to allow application of (\ref{eqn:differential}). Using (\ref{eqn:differential}, \ref{eqn:mixed differential}) with $i$ odd and $j$ odd,
\begin{align*}
\partial u_{i,j} &= \partial \smashoperator{\sum_{k=-1}^{m-i}} x_{i+k}y_{j-k} = \partial \smashoperator{\sum_{\substack{k=0 \\ k\text{ even}}}^{m-i-1}} x_{i+k}y_{j-k} + \partial \smashoperator{\sum_{\substack{k=-1 \\ k\text{ odd}}}^{m-i}} x_{i+k}y_{j-k} \\
&= \smashoperator{\sum_{\substack{k=0 \\ k\text{ even}}}^{m-i-1}} (x_{i+k-1}y_{j-k} + x_{i+k+1}y_{j-k} + x_{i+k}y_{j-k-1} + x_{i+k}y_{j-k+1}) + 0 \\
&= \smashoperator{\sum_{\substack{k=0 \\ k\text{ even}}}^{m-i-1}} (x_{i+k-1}y_{j-k} + x_{i+k}y_{j-k-1}) + \smashoperator{\sum_{\substack{k=0 \\ k\text{ even}}}^{m-i-1}} (x_{i+k+1}y_{j-k} + x_{i+k}y_{j-k+1}) \\
&= -(x_{i-2}y_{j+1} + x_my_{j+i-m-1}) + \smashoperator{\sum_{k=-1}^{m-(i-1)}} x_{i-1+k}y_{j-k} + \smashoperator{\sum_{k=-1}^{m-(i+1)}} x_{i+1+k}y_{j-k} \\
&= v_{i-2,j+1} + u_{i-1,j} + u_{i+1,j} \displaybreak \\
\partial v_{i,j+1} &= \partial(x_iy_{j+1} + x_my_{j+1+i-m}) \\
&= (x_{i-1}y_{j+1} + x_{i+1}y_{j+1}) + (x_my_{j+1+i-m-1} + x_my_{j+1+i-m+1}) \\
&= (x_{i-1}y_{j+1} + x_my_{j+1+i-1-m}) + (x_{i+1}y_{j+1} + x_my_{j+1+i+1-m}) \\
&= v_{i-1,j+1} + v_{i+1,j+1}.
\end{align*}
Abbreviating $u_i = u_{i,j}$ and $v_i = v_{i,j+1}$ for convenience, we have shown that
\begin{equation}
\partial u_i = \begin{cases}[1] v_{i-2} + u_{i-1} + u_{i+1} &i \text{ odd} \\ v_{i-2} &i \text{ even} \end{cases} \qquad \text{and} \qquad \partial v_i = \begin{cases}[1] v_{i-1} + v_{i+1} &i \text{ odd} \\ 0 &i \text{ even}. \end{cases}
\end{equation}
Notice $u_m = v_{m-1}$, $v_m = 0$, and $\partial(v_{-1} + u_0) = v_0$. That is, $R'_j$ is a basis for two staircases, with arrows joining the endpoints $\{v_{-1} + u_0, u_{m-1}\}$ to $\{v_0, v_{m-1}\}$. (Note the diagonal arrows from $u_i$ to $v_{i-2}$ for $i > 1$ ensure $\partial^2 = 0$.) As there are no outgoing arrows from $S$ to $\{x_iy_j\} \backslash S$, $B'_{i,j}$ to $\{x_iy_j\} \backslash B_{i,j}$, or $R'_j$ to $\{x_iy_j\} \backslash R_j$, neither are there incoming arrows to $B'_{i,j}$ or $R'_j$. We conclude that each $B'_{i,j}$ and $R'_j$ is a basis for an acyclic subcomplex that splits off as a direct summand.
\end{proof}

\section{Ordering of Staircase Complex Classes} \label{sec:ordering}

We now study the ordering on $\cF_\text{alg}$. We take the tensor product of a complex with $n$ times the dual of another, both reduced with $\varepsilon = 1$, then partially vertically and horizontally simplify the basis of the product to determine its $\varepsilon$ value.

\begin{lemma}[{\cite[Lemmas 6.3 and 6.4]{Homsmooth}}] \label{lem:order i}
Let $a_i,\, b_j > 0$. If $b_1 > a_1$ or if $b_1 = a_1$ and $b_2 < a_2$, then
\[ [a_1, a_2, \ldots, a_m] \gg [b_1, b_2, \ldots, b_n]. \]
\end{lemma}

\begin{lemma} \label{lem:order j}
Let $a,\, c > 0$ and $d,\, p,\, q \geq 0$. If $q > p$ or if $q = p$ and $d < c$, then
\[ [(1, a)^p\!, 1, a+c] \gg [(1, a)^q\!, 1, a+d]. \]
\end{lemma}

\begin{proof}
Set $m = 2p+2$ and $n = 2q+2$. For $C \in ((1, a)^p\!, 1, a+c)$ with basis $\{x_i\}$, $i = 0, \ldots, 2m$, and $D_r \in ((1, a)^q\!, 1, a+d)^r$ with basis $\{y_j\}$, $j = 0, \ldots, 2nr$, we prove $\varepsilon(C \otimes rD_1^*) = 1$ for all $r \in \N$. By inductive use of Lemma \ref{lem:polygon}, $rD_1 \sim_\varepsilon D_r$, and so we may work with the simpler complexes $D_r$, i.e., we show $\varepsilon(C \otimes D_r^*) = 1$.

By \cite[Section 3]{Homsmooth}, Definition \ref{def:epsilon} (1) is equivalent to finding a basis with element $u_0$ that is the distinguished element of some vertically simplified basis for $C \otimes D_r^*$ and in the image of the horizontal differential. Define the basis element sums
\begin{equation} \label{eqn:cob u}
u_i = \smashoperator{\sum_{k=0}^{m-i}} x_{i+k}y_k^*.
\end{equation}
Supposing $q \geq p$, we may produce the new basis $x_iy_0^* \to u_i$ for all $i < m$. Clearly $u_0 = (x_0y_0^*)'$ has no incoming vertical arrows since $x_0y_0^*$ has none, so it suffices to show $\partial^\text{vert}u_0 = 0$ and $\partial^\text{horz}u_1 = u_0$. That is, we need only partially simplify the basis.

First, we check that $x_iy_0^* \to u_i$ respects the filtration. Let $0 < k < m-i$. By (\ref{eqn:fl}, \ref{eqn:mixed fl}) with $i$ even,
\begin{align*}
\fl(x_{i+k}y_k^*) &= \begin{cases}[1] (\fl(x_{i+k-1}) + (1, 0)) + (\fl(y_{k-1}^*) - (1, 0)) &k \text{ odd} \\ (\fl(x_{i+k-1}) - (0, a)) + (\fl(y_{k-1}^*) + (0, a)) &k \text{ even} \end{cases} \\
&= \fl(x_{i+k-1}y_{k-1}^*) \\
\fl(x_my_{m-i}^*) &= (\fl(x_{m-1}) - (0, a+c)) + (\fl(y_{m-i-1}^*) + (0, a)) \\
&= \fl(x_{m-1}y_{m-i-1}^*) - (0, c),\  i \neq 0 \text{ if } q = p
\end{align*}
and $\fl(x_my_m^*) = \fl(x_{m-1}y_{m-1}^*) - (0, c-d)$ if $q = p$. For $i$ odd,
\begin{align*}
\fl(x_{i+k}y_k^*) &= \begin{cases}[1] (\fl(x_{i+k-1}) - (0, a)) + (\fl(y_{k-1}^*) - (1, 0)) &k \text{ odd} \\ (\fl(x_{i+k-1}) + (1, 0)) + (\fl(y_{k-1}^*) + (0, a)) &k \text{ even} \end{cases} \\
&= \fl(x_{i+k-1}y_{k-1}^*) + (-1)^k(1, a) \\
\fl(x_my_{m-i}^*) &= (\fl(x_{m-1}) - (0, a+c)) + (\fl(y_{m-i-1}^*) - (1, 0)) \\
&= \fl(x_{m-1}y_{m-i-1}^*) - (1, a+c).
\end{align*}
Letting $e_i = i \bmod 2$, it follows that
\begin{equation} \label{eqn:fl u}
\fl(x_iy_0^*) = \begin{cases}[1] \fl(x_{i+k}y_k^*) + e_{ik}(1, a) \\ \fl(x_my_{m-i}^*) + (0, c) + e_i(1, a) &i \neq 0 \text{ if } q = p \\ \fl(x_my_m^*) + (0, c-d) &i = 0,\, q = p. \end{cases}
\end{equation}
We require $\fl(x_iy_0^*) \geq \fl(x_{i+k}y_k^*)$ for each $k$ summed over in $u_i$. By (\ref{eqn:fl u}), this is equivalent to the condition that $c \geq 0$ and that $d \leq c$ if $q = p$, which is true by hypothesis.

Second, we check differentials. Using (\ref{eqn:differential}, \ref{eqn:mixed differential}) with $i$ even,
\begin{align*}
\partial u_i &= \partial \smashoperator{\sum_{k=0}^{m-i}} x_{i+k}y_k^* = \partial(x_iy_0^*) + \partial \smashoperator{\sum_{\substack{k=1 \\ k\text{ odd}}}^{m-i-1}} (x_{i+k}y_k^* + x_{i+k+1}y_{k+1}^*) \\
&= x_iy_1^* + \smashoperator{\sum_{\substack{k=1 \\ k\text{ odd}}}^{m-i-1}} ((x_{i+k-1}y_k^* + x_{i+k+1}y_k^*) + (x_{i+k+1}y_{k+2}^* + x_{i+k+1}y_k^*)) \\
&= x_iy_1^* + \smashoperator{\sum_{\substack{k=1 \\ k\text{ odd}}}^{m-i-1}} (x_{i+k-1}y_k^* + x_{i+k+1}y_{k+2}^*) + \smashoperator{\sum_{\substack{k=1 \\ k\text{ odd}}}^{m-i-1}} 2x_{i+k+1}y_k^* \\
&= x_my_{m-i+1}^* + \smashoperator{\sum_{\substack{k=1 \\ k\text{ odd}}}^{m-i-1}} 2x_{i+k-1}y_k^* + \smashoperator{\sum_{\substack{k=1 \\ k\text{ odd}}}^{m-i-1}} 2x_{i+k+1}y_k^* \\
&= x_my_{m-i+1}^*.
\end{align*}
For $i$ odd,
\begin{align*}
\partial u_i &= \partial \smashoperator{\sum_{k=0}^{m-i}} x_{i+k}y_k^* = \partial(x_iy_0^*) + \partial \smashoperator{\sum_{\substack{k=2 \\ k\text{ even}}}^{m-i-1}} x_{i+k}y_k^* + \partial \smashoperator{\sum_{\substack{k=1 \\ k\text{ odd}}}^{m-i}} x_{i+k}y_k^* \\
&= \smashoperator{\sum_{\substack{k=0 \\ k\text{ even}}}^{m-i-1}} (x_{i+k-1}y_k^* + x_{i+k+1}y_k^* + x_{i+k}y_{k+1}^*) + \smashoperator{\sum_{\substack{k=2 \\ k\text{ even}}}^{m-i-1}} x_{i+k}y_{k-1}^* + 0 \\
&= \smashoperator{\sum_{\substack{k=0 \\ k\text{ even}}}^{m-i-1}} (x_{i+k-1}y_k^* + x_{i+k}y_{k+1}^*) + x_{i+1}y_0^* + \smashoperator{\sum_{\substack{k=2 \\ k\text{ even}}}^{m-i-1}} (x_{i+k+1}y_k^* + x_{i+k}y_{k-1}^*) \\
&= -x_my_{m-i+1}^* + \smashoperator{\sum_{k=0}^{m-(i-1)}} x_{i-1+k}y_k^* + \smashoperator{\sum_{k=0}^{m-(i+1)}} x_{i+1+k}y_k^* \\
&= x_my_{m-i+1}^* + u_{i-1} + u_{i+1}.
\end{align*}
We have shown that
\begin{equation} \label{eqn:differential u}
\partial u_i = \begin{cases}[1] x_my_{m-i+1}^* + u_{i-1} + u_{i+1} &i \text{ odd} \\ x_my_{m-i+1}^* &i \text{ even}. \end{cases}
\end{equation}

Now suppose that $c > 0$ and that $d < c$ if $q = p$. Letting $\tau = \tau(C \otimes D_r^*)$, note that $\fl(u_0) = (0, \tau)$, $\fl(u_1) = (1, \tau)$, and $\fl(u_2) = (1, \tau-a)$ if $p > 0$ or $(1, \tau-a-c)$ if $p = 0$. Using (\ref{eqn:fl u}), also note that $\fl(x_my_m^*) = (0, \tau-c)$ if $q > p$ or $(0, \tau-c+d)$ if $q = p$. By (\ref{eqn:differential u}), we have $\partial u_0 = x_my_{m+1}^*$. Notice that the $i$-coordinate of $u_0$ is $0$, while the $i$-coordinate of $x_my_{m+1}^*$ is less than $0$ (because $x_my_m^*$ is at $i = 0$); thus $\partial^\text{vert}u_0 = 0$. Again by (\ref{eqn:differential u}), we have $\partial u_1 = x_my_m^* + u_0 + u_2$. Notice that the $j$-coordinate of $u_1$ and $u_0$ is $\tau$, while the $j$-coordinate of both $u_2$ and $x_my_m^*$ is less than $\tau$; thus $\partial^\text{horz}u_1 = u_0$.
\end{proof}

It is easily seen that $\{u_i \mid i < m\} \cup \{x_my_j^* \mid j \leq nr\}$ with its reflection in the above proof spans a subcomplex isomorphic to the reduced complex $P_r \in ((1, a)^p\!, 1, a+c, ((-1, -a)^q\!, -1, -a-d)^r)$. This set belongs to a fully vertically and horizontally simplified basis for $C \otimes D_r^*$, from which one obtains $[C] - r[D_1] = [P_r]$, which we leave as an exercise for the dedicated reader.

\section{Floer Complexes of Selected $L$-space Knots} \label{sec:floer}

In this section, we find formulas for the Alexander polynomials of certain cables of torus knots. The knots we consider are all $L$-space knots, and thus the computations of the Alexander polynomials in fact gives us the knot Floer complexes of these knots.

This first lemma involves taking a well-known formula for the Alexander polynomial of iterated torus knots and grouping the terms in order to simply the expression, for example by noticing telescoping sums.

\begin{lemma} \label{lem:cablepoly}
The Alexander polynomial of the $(p, pm(m-1)+1)$-cable of the $(m, m+1)$-torus knot is
\[ \Delta_{T_{m,m+1;m,pm(m-1)+1}}(t) = \smashoperator{\sum_{i=0}^{pm(m-1)}} t^{ip} - t \left( \sum_{i=0}^{p-1} t^{i(pm^2-pm+1)} \right) \left( \smashoperator[r]{\sum_{j=0}^{m-2}} t^{jpm} \left( \sum_{k=0}^j t^{kp} + \smashoperator{\sum_{k=j+1}^{m-1}} t^{kp-1} \right) \right). \]
Similarly, the Alexander polynomial of the $(p, pm(m-1)-1)$-cable of the $(m, m+1)$-torus knot is
\[ \Delta_{T_{m,m+1;p,pm(m-1)-1}}(t) = -t \smashoperator{\sum_{i=0}^{pm(m-1)-2}} t^{ip} + \left( \sum_{i=0}^{p-1} t^{i(pm^2-pm-1)} \right) \left( \smashoperator[r]{\sum_{j=0}^{m-2}} t^{jpm} \left( \sum_{k=0}^j t^{kp} + \smashoperator{\sum_{k=j+1}^{m-1}} t^{kp+1} \right) \right). \]
\end{lemma}

\begin{proof}
We know that
\[ \Delta_{T_{m,m+1;p,pm(m-1)+1}}(t) = \frac{(t^{pm(m+1)} - 1)(t^p - 1)}{(t^{pm} - 1)(t^{p(m+1)} - 1)} \cdot \frac{(t^{p(pm(m-1)+1)} - 1)(t - 1)}{(t^p - 1)(t^{pm(m-1)+1} - 1)}. \]
Let
\[ P(t) = \smashoperator{\sum_{i=0}^{pm(m-1)}} t^{ip} - t \left( \sum_{i=0}^{p-1} t^{i(pm^2-pm+1)} \right) \left( \smashoperator[r]{\sum_{j=0}^{m-2}} t^{jpm} \left( \sum_{k=0}^j t^{kp} + \smashoperator{\sum_{k=j+1}^{m-1}} t^{kp-1} \right) \right). \]
To prove the lemma, we will show that
\begin{equation} \label{eqn:wanttoshow}
(t^{pm} - 1)(t^{p(m+1)} - 1)(t^p - 1)(t^{pm(m-1)+1} - 1) \cdot P(t) = (t^{pm(m+1)} - 1)(t^p - 1)(t^{p(pm(m-1)+1)} - 1)(t - 1).
\end{equation}

We first consider
\begin{equation} \label{eqn:telescope1}
(t^{pm} - 1)(t^{p(m+1)} - 1)(t^p - 1)(t^{pm(m-1)+1} - 1) \smashoperator{\sum_{i=0}^{pm(m-1)}} t^{ip}.
\end{equation}
Note the telescoping sum
\[ (t^p - 1) \smashoperator{\sum_{i=0}^{pm(m-1)}} t^{ip} = t^{p^2m^2-p^2m+p} - 1. \]
Hence (\ref{eqn:telescope1}) is equal to
\begin{equation} \label{eqn:part1}
(t^{pm} - 1)(t^{p(m+1)} - 1)(t^{pm(m-1)+1} - 1)(t^{p^2m^2-p^2m+p} - 1).
\end{equation}
Next, we consider
\begin{flalign*}
(t^{pm} - 1)(t^{p(m+1)} - 1)(t^p - 1)&&
\end{flalign*}
\vspace{-30pt}
\begin{flalign*}
&&{} \cdot (t^{pm(m-1)+1} - 1)(-t) \left( \sum_{i=0}^{p-1} t^{i(pm^2-pm+1)} \right) \left( \smashoperator[r]{\sum_{j=0}^{m-2}} t^{jpm} \left( \sum_{k=0}^j t^{kp} + \smashoperator{\sum_{k=j+1}^{m-1}} t^{kp-1} \right) \right).
\end{flalign*}
We first notice that
\begin{equation} \label{eqn:part2}
(t^{pm(m-1)+1} - 1)(-t) \left( \sum_{i=0}^{p-1} t^{i(pm^2-pm+1)} \right) = -t \cdot (t^{p^2m^2-p^2m+p} - 1).
\end{equation}
Grouping the remaining factors yields
\begin{align*} \label{eqn:part3}
&(t^{pm} - 1)(t^{p(m+1)} - 1)(t^p - 1) \left( \smashoperator[r]{\sum_{j=0}^{m-2}} t^{jpm} \left( \sum_{k=0}^j t^{kp} + \smashoperator{\sum_{k=j+1}^{m-1}} t^{kp-1} \right) \right) \\
&= (t^{pm} - 1)(t^{p(m+1)} - 1) \left( \smashoperator[r]{\sum_{j=0}^{m-2}} t^{jpm} \left( t^{jp+p} - 1 + t^{pm-1} - t^{jp+p-1} \right) \right) \\
&= (t^{pm} - 1)(t^{p(m+1)} - 1) \left( \smashoperator[r]{\sum_{j=0}^{m-2}} t^{jpm} \left( t^{pm-1} - 1 + t^{jp+p-1}(t - 1) \right) \right) \\
&= (t^{pm} - 1)(t^{p(m+1)} - 1) \left( (t^{pm-1} - 1) \smashoperator{\sum_{j=0}^{m-2}} t^{jpm} + t^{p-1}(t - 1) \smashoperator{\sum_{j=0}^{m-2}} t^{jp(m+1)} \right) \\
&= (t^{p(m+1)} - 1)(t^{pm-1} - 1)(t^{pm^2-pm} - 1) + (t^{pm} - 1)t^{p-1}(t - 1)(t^{pm^2-p} - 1). \numberthis
\end{align*}
We have shown that the lefthand side of (\ref{eqn:wanttoshow}) is equal to
\[ (\ref{eqn:part1}) + (\ref{eqn:part2}) \cdot (\ref{eqn:part3}), \]
where we have used telescoping sums to eliminate the summation from each of those expressions. It is straightforward to verify that this is equal to the righthand side of (\ref{eqn:wanttoshow}).

The calculation for $\Delta_{T_{m,m+1;p,pm(m-1)-1}}(t)$ follows in a similar manner.
\end{proof}

The following lemma is essentially a restatement of Lemma \ref{lem:cablepoly} in our language for the $\varepsilon$-equivalence classes of staircase complexes, i.e., knot Floer complexes of $L$-space knots.

\begin{lemma} \label{lem:cablestaircase}
Let $2g = m(m-1)$ and
\begin{equation} \label{eqn:cableseq}
\begin{aligned}
(x_s)_{s=1}^{4gp} &= (((i, p-i)^j\!, (i-1, p-i+1)^{m-j})_{j=1}^{m-1})_{i=1}^p \\
(y_s)_{s=1}^{4gp} &= (((p-i, i)^j\!, (p-i+1, i-1)^{m-j})_{j=1}^{m-1})_{i=1}^p.
\end{aligned}
\end{equation}
Then for positive $p$ and $m$,
\[ \llbracket T_{m,m+1;p,pm(m-1)+1} \rrbracket = [(x_s)_{s=1}^{2gp}] \qquad \text{and} \qquad \llbracket T_{m,m+1;p,pm(m-1)-1} \rrbracket = [(y_s)_{s=2}^{2gp}]. \]
\end{lemma}

\begin{proof}
We know that $\CFKi(T_{m,m+1;p,pm(m-1)\pm1})$ are staircases. We translate backwards from the sequences of step lengths to the corresponding summations. It will be helpful to rewrite the sequences as
\begin{equation} \label{eqn:cableseqk}
\begin{aligned}
(x_{2s+1}, x_{2s+2})_{s=0}^{2gp-1} &= (((i+1, p-i-1)_{k=0}^j, (i, p-i)_{k=j+1}^{m-1})_{j=0}^{m-2})_{i=0}^{p-1} \\
(y_{2s+1}, y_{2s+2})_{s=0}^{2gp-1} &= (((p-i-1, i+1)_{k=0}^j, (p-i, i)_{k=j+1}^{m-1})_{j=0}^{m-2})_{i=0}^{p-1}
\end{aligned}
\end{equation}
and to note that each integer $s$ can be written uniquely as
\begin{equation} \label{eqn:sindex}
s = i(m-1)m+jm+k \quad \text{where } 0 \leq j < m-1 \text{ and } 0 \leq k < m.
\end{equation}

Let $n_0 = 0$, and inductively define $n_i = n_{i-1} + x_i$ for $i > 0$. More generally, it follows that $n_i = n_j + \sum_{k=j+1}^i x_k$ for $i > j$. We claim that $\sum_{i=0}^{4gp} (-1)^it^{n_i}$ equals $\Delta_{T_{m,m+1;p,pm(m-1)+1}}(t)$, as given by the first polynomial in the statement of Lemma \ref{lem:cablepoly}. Suppose that $n_{2s} = sp$. For $i$, $j$, and $k$ given as in (\ref{eqn:sindex}), i.e., $i = \lfloor s/((m-1)m) \rfloor$, $j = \lfloor s/m \rfloor - i(m-1)$, and $k = s - i(m-1)m - jm$, it is clear from (\ref{eqn:cableseqk}) that
\[ n_{2(s+1)} = n_{2s} + x_{2s+1} + x_{2s+2} = \left. \begin{cases}[1] sp + i + 1 + p - i - 1 &k \leq j \\ sp + i + p - i &k > j \end{cases} \right\rbrace = (s+1)p. \]
It follows that $n_{2s} = sp$ for each $s$. That is,
\[ \sum_{s=0}^{2gp} t^{n_{2s}} = \sum_{s=0}^{2gp} t^{sp} = \smashoperator{\sum_{i=0}^{pm(m-1)}} t^{ip}, \]
which are all of the terms with positive coefficient. Next, we have from (\ref{eqn:cableseqk}) that
\[ n_{2s+1} = n_{2s} + x_{2s+1} = \begin{cases}[1] sp + i + 1 &k \leq j \\ sp + i &k > j. \end{cases} \]
It follows that
\[ \smashoperator{\sum_{s=0}^{2gp-1}} t^{n_{2s+1}} = \smashoperator{\sum_{\substack{s=0 \\ k\leq j}}^{2gp-1}} t^{sp+i+1} + \smashoperator{\sum_{\substack{s=0 \\ k>j}}^{2gp-1}} t^{sp+i} = \sum_{i=0}^{p-1} \sum_{j=0}^{m-2} \sum_{k=0}^j t^{sp+i+1} + \sum_{i=0}^{p-1} \smashoperator[r]{\sum_{j=0}^{m-2}} \smashoperator[r]{\sum_{k=j+1}^{m-1}} t^{sp+i}, \]
where $s = i(m-1)m+jm+k$, which are all of the terms with negative coefficient. It is straightforward to verify that $\sum_{i=0}^{4gp} (-1)^it^{n_i} = \sum_{s=0}^{2gp} t^{n_{2s}} - \sum_{s=0}^{2gp-1} t^{n_{2s+1}}$ is equal to $\Delta_{T_{m,m+1;p,pm(m-1)+1}}(t)$.

Now let $n_0 = 0$, and inductively define $n_i = n_{i-1} + y_i$ for $i > 0$ so that $n_i = n_j + \sum_{k=j+1}^i y_k$ for $i > j$. Note that the sequence for $\llbracket T_{m,m+1;p,pm(m-1)-1} \rrbracket$ begins with $y_2$, not $y_1$. We claim that $\sum_{i=1}^{4gp-1} (-1)^it^{n_i}$ equals $\Delta_{T_{m,m+1;p,pm(m-1)-1}}(t)$, as given by the second polynomial in the statement of Lemma \ref{lem:cablepoly}, up to a monomial factor. We find $\sum_{i=0}^{4gp} (-1)^it^{n_i}$ in a similar manner as before. The terms with positive coefficient $\sum_{s=0}^{2gp} t^{n_{2s}}$ are the same, and the terms with negative coefficient are
\[ \smashoperator{\sum_{s=0}^{2gp-1}} t^{n_{2s+1}} = \sum_{i=0}^{p-1} \sum_{j=0}^{m-2} \sum_{k=0}^j t^{sp+p-i-1} + \sum_{i=0}^{p-1} \smashoperator[r]{\sum_{j=0}^{m-2}} \smashoperator[r]{\sum_{k=j+1}^{m-1}} t^{sp+p-i}, \]
where $s = i(m-1)m+jm+k$. It is straightforward to verify that $\sum_{i=1}^{4gp-1} (-1)^it^{n_i}$ = $\sum_{s=1}^{2gp-1} t^{n_{2s}} - \sum_{s=0}^{2gp-1} t^{n_{2s+1}}$ is equal to $-t^{p-1}\Delta_{T_{m,m+1;p,pm(m-1)-1}}(t)$.
\end{proof}

\noindent The corollary below follows by taking $p = 1$ (cf. \cite[Proposition 6.1]{HeddenLR}).

\begin{corollary} \label{cor:torusstaircase}
Let
\begin{equation} \label{eqn:torusseq}
(t_s)_{s=1}^{2m-2} = (j, m-j)_{j=1}^{m-1}.
\end{equation}
Then for positive $m$,
\[ \llbracket T_{m,m+1} \rrbracket = [(t_s)_{s=1}^{m-1}]. \]
\end{corollary}

\begin{proof}
Take $p = 1$ in $(x_s)$, since $T_{m,m+1;1,m(m-1)+1} = T_{m,m+1}$. Then $[(x_s)_{s=1}^{2gp}]$ simplifies as
\[ [(x_s)_{s=1}^{2g}] = [((1, 0)^j\!, (0, 1)^{m-j})_{j=1}^{(m-1)/2}] = [(j, m-j)_{j=1}^{(m-1)/2}] \]
for $m$ odd, and similarly as $[(x_s)_{s=1}^{2g}] = [(j, m-j)_{j=1}^{(m-2)/2}\!, m/2]$ for $m$ even. It follows that $[(x_s)_{s=1}^{2gp}] = [(t_s)_{s=1}^{m-1}]$.
\end{proof}

Note that $(x_s)$ and $(y_s)$ have $0$ entries where the terms with positive coefficient cancel those with negative coefficient (see proof of Lemma \ref{lem:cablestaircase}). This occurs when $i = 1$ or $p$ in (\ref{eqn:cableseq}). We can therefore simplify the complexes $[(x_s)_{s=1}^{2gp}]$ and $[(y_s)_{s=2}^{2gp}]$ as follows:
\begin{align*}
[(x_s)_{s=1}^{2gp}] &= [((1, p-1)^j\!, (0, p)^{m-j})_{j=1}^{m-1}\!, (x_s)_{s=4g+1}^{2gp}] \\
&= [((1, p-1)^{j-1}\!, 1, p-1, 0, p(m-j))_{j=1}^{m-1}\!, (x_s)_{s=4g+1}^{2gp}] \\
&= [((1, p-1)^{j-1}\!, 1, p(m-j+1)-1)_{j=1}^{m-1}\!, (x_s)_{s=4g+1}^{2gp}] \\
[(y_s)_{s=2}^{2gp}] &= [1, (p, 0)^{m-1}\!, ((p-1, 1)^j\!, (p, 0)^{m-j})_{j=2}^{m-1}\!, p-2, (y_s)_{s=4g+2}^{2gp}] \\
&= [1, p(m-1), 0, (p-1, (1, p-1)^{j-1}\!, 1, p(m-j), 0)_{j=2}^{m-1}\!, p-2, (y_s)_{s=4g+2}^{2gp}] \\
&= [((1, p-1)^{j-1}\!, 1, p(m-j), 0, p-1)_{j=1}^{m-2}\!, (1, p-1)^{m-2}\!, 1, p, 0, p-2, (y_s)_{s=4g+2}^{2gp}] \\
&= [((1, p-1)^{j-1}\!, 1, p(m-j+1)-1)_{j=1}^{m-2}\!, (1, p-1)^{m-2}\!, 1, 2p-2, (y_s)_{s=4g+2}^{2gp}].
\end{align*}
Changing index by letting $j \to j+1$, we can write
\begin{equation} \label{eqn:initialstairs}
\begin{aligned}{}
[(x_s)_{s=1}^{2gp}] &= [((1, p-1)^j\!, 1, p(m-j)-1)_{j=0}^{m-2}\!, (x_s)_{s=4g+1}^{2gp}] \\
[(y_s)_{s=2}^{2gp}] &= [((1, p-1)^j\!, 1, p(m-j)-1)_{j=0}^{m-3}\!, (1, p-1)^{m-2}\!, 1, 2p-2, (y_s)_{s=4g+2}^{2gp}].
\end{aligned}
\end{equation}

\section{Archimedean Equivalence Classes of $\cF$} \label{sec:archimedean}

We can now combine the results of Sections \ref{sec:tensor}, \ref{sec:ordering}, and \ref{sec:floer} to find new Archimedean equivalence classes in $\cF$. In the following lemma, we define the knots that will be at the heart of Theorem \ref{thm:ordertype}.

\begin{lemma}
Set $p = i+1$ and $m = |j|+3$ for $i, j \in \Z$. Define
\[ K_{(i, j)} = \begin{cases} T_{m,m+1;p,pm(m-1)+1} \cs -T_{pm,pm+1} &i > 0,\, j \geq 0 \\ T_{m,m+1;p,pm(m-1)+1} \cs -T_{m,m+1;p,pm(m-1)-1} &i > 0,\, j < 0 \\ T_{m,m+1} \cs -T_{2,3;\lceil m/2 \rceil,2\lfloor m/2 \rfloor+1} &i = 0,\, j \geq 0. \end{cases} \]
Then $\llbracket K_{(0, 0)} \rrbracket = 0$ and for $(i, j) > (0, 0)$,
\[ \llbracket K_{(i, j)} \rrbracket \sim_\textup{Ar} \begin{cases}{} [1, i, 1, 2i+1+j(i+1)] &j \geq 0 \\ [(1, i)^{-j}, 1, i, 1, 2i+1] &j \leq 0. \end{cases} \]
\end{lemma}

\begin{proof}
Making inductive use of Lemma \ref{lem:box} on $[(t_s)_{s=1}^{m-1}]$ of Corollary \ref{cor:torusstaircase}, we obtain
\begin{equation} \label{eqn:torusdecomp}
\llbracket T_{m,m+1} \rrbracket = \sum_{k=1}^n [k, m-k] + [(t_s)_{s=2n+1}^{m-1}]
\end{equation}
for any $0 \leq n < \lceil m/2 \rceil$. Using (\ref{eqn:initialstairs}), we can apply Lemma \ref{lem:box} to $[(x_s)_{s=1}^{2pg}]$ and $[(y_s)_{s=2}^{2pg}]$ of Lemma \ref{lem:cablestaircase} to obtain
\begin{equation} \label{eqn:cabledecomp1}
\begin{aligned}
\llbracket T_{m,m+1;p,pm(m-1)+1} \rrbracket &= [(x_s)_{s=1}^{4g}] + [(x_s)_{s=4g+1}^{2gp}] \\
\llbracket T_{m,m+1;p,pm(m-1)-1} \rrbracket &= [(y_s)_{s=2}^{4g+1}] + [(y_s)_{s=4g+2}^{2gp}],
\end{aligned}
\end{equation}
where
\begin{align*}
[(x_s)_{s=1}^{4g}] &= [((1, p-1)^k\!, 1, p(m-k)-1)_{k=0}^{m-2}] \\
[(y_s)_{s=2}^{4g+1}] &= [((1, p-1)^k\!, 1, p(m-k)-1)_{k=0}^{m-3}\!, (1, p-1)^{m-2}\!, 1, 2p-2].
\end{align*}
In addition, inductive use of Lemma \ref{lem:polygon} gives
\begin{equation} \label{eqn:cabledecomp2}
\begin{aligned}{}
[(x_s)_{s=1}^{4g}] &= \smashoperator{\sum_{k=0}^{m-2}} [(1, p-1)^k\!, 1, p(m-k)-1] \\
[(y_s)_{s=2}^{4g+1}] &= \smashoperator{\sum_{k=0}^{m-3}} [(1, p-1)^k\!, 1, p(m-k)-1] + [(1, p-1)^{m-2}\!, 1, 2p-2].
\end{aligned}
\end{equation}
We use the class decompositions (\ref{eqn:torusdecomp}, \ref{eqn:cabledecomp1}, \ref{eqn:cabledecomp2}) along with the sequence definitions (\ref{eqn:cableseq}, \ref{eqn:torusseq}) in the following calculations.

\textbf{\boldmath Case 1: $i > 0,\, j \geq 0$.} We have $\llbracket K_{(i, j)} \rrbracket = \llbracket T_{m,m+1;p,pm(m-1)+1} \rrbracket - \llbracket T_{pm,pm+1} \rrbracket$. Hence $\llbracket K_{(i, j)} \rrbracket$ is given by
\begin{align*}
&\smashoperator{\sum_{k=0}^{m-2}} [(1, p-1)^k\!, 1, p(m-k)-1] + [(x_s)_{s=4g+1}^{2gp}] - [1, pm-1] - [(t_s)_{s=3}^{pm-1}] \\
&= \smashoperator{\sum_{k=1}^{m-2}} [(1, p-1)^k\!, 1, p(m-k)-1] + [2, (x_s)_{s=4g+2}^{2gp}] - [2, (t_s)_{s=4}^{pm-1}],
\end{align*}
where the $[1, pm-1]$ term cancels. By Lemmas \ref{lem:order i} and \ref{lem:order j}, $[1, p-1, 1, p(m-1)-1] \gg$ the other remaining terms so that $\llbracket K_{(i, j)} \rrbracket \sim_\text{Ar} [1, p-1, 1, p(m-1)-1]$.

\textbf{\boldmath Case 2: $i > 0,\, j < 0$.} We have $\llbracket K_{(i, j)} \rrbracket = \llbracket T_{m,m+1;p,pm(m-1)+1} \rrbracket - \llbracket T_{m,m+1;p,pm(m-1)-1} \rrbracket$. Hence $\llbracket K_{(1, j)} \rrbracket = [(1, 1)^{m-2}\!, 1, 3] - [(1, 1)^{m-2}\!, 1, 2] \sim_\text{Ar} [(1, 1)^{m-2}\!, 1, 3]$, and for $i \geq 2$, $\llbracket K_{(i, j)} \rrbracket$ is given by
\begin{align*}
&\smashoperator{\sum_{k=0}^{m-2}} [(1, p-1)^k\!, 1, p(m-k)-1] + [(x_s)_{s=4g+1}^{2pg}] - \smashoperator{\sum_{k=0}^{m-3}} [(1, p-1)^k\!, 1, p(m-k)-1] \\
&\qquad - [(1, p-1)^{m-2}\!, 1, 2p-2] - [(y_s)_{s=4g+2}^{2pg}] \\
&= [(1, p-1)^{m-2}\!, 1, 2p-1]  + [2, (x_s)_{s=4g+2}^{2pg}] - [(1, p-1)^{m-2}\!, 1, 2p-2] - [2, (y_s)_{s=4g+3}^{2pg}],
\end{align*}
where each $[(1, p-1)^k\!, 1, p(m-k)-1]$ term for $0 \leq k \leq m-3$ cancels. By Lemmas \ref{lem:order i} and \ref{lem:order j}, $[(1, p-1)^{m-2}\!, 1, 2p-1] \gg$ the other remaining terms so that $\llbracket K_{(i, j)} \rrbracket \sim_\text{Ar} [(1, p-1)^{m-2}\!, 1, 2p-1]$.

\textbf{\boldmath Case 3: $i = 0,\, j > 0$.} Letting $n = \lceil m/2 \rceil$ and noting $[1, 0, 1, 1+j] = [2, 1+j]$, we can rephrase the lemma for this case as
\[ K_{(0, 2n-(7\mp1)/2)} = T_{2n,2n\pm1} \cs -T_{2,3;n,2n\pm1} \implies \llbracket K_{(0, j)} \rrbracket \sim_\text{Ar} [2, 1+j]. \]
For $j$ odd, we have $\llbracket K_{(0, j)} \rrbracket = \llbracket T_{2n,2n+1} \rrbracket - \llbracket T_{2,3;n,2n+1} \rrbracket$. Hence $\llbracket K_{(0, 1)} \rrbracket = [1, 3] + [2] - [1, 3] = [2] \sim_\text{Ar} 2[2] = [2, 2]$, and for $j \geq 3$, $\llbracket K_{(0, j)} \rrbracket$ is given by
\begin{align*}
&\sum_{k=1}^{2} [k, 2n-k] + [(t_s)_{s=5}^{2n-1}] - [1, 2n-1] - [(x_s)_{s=5}^{2n}] \\
&= [2, 2n-2] + [3, (t_s)_{s=6}^{2n-1}] - [2, n-2, (x_s)_{s=7}^{2n}],
\end{align*}
where the $[1, 2n-1]$ term cancels. By Lemma \ref{lem:order i}, $[2, 2n-2] \gg$ the other remaining terms so that $\llbracket K_{(0, j)} \rrbracket \sim_\text{Ar} [2, 2n-2] = [2, m-2]$. For $j$ even, we similarly find $\llbracket K_{(0, j)} \rrbracket \sim_\text{Ar} [2, 2n-3] = [2, m-2]$.

The result for $(i, j) > (0, 0)$ follows by substituting $(i+1)$ for $p$ and $(|j|+3)$ for $m$ in each case, and $\llbracket K_{(0, 0)} \rrbracket = \llbracket T_{3,4} \rrbracket - \llbracket T_{2,3;2,3} \rrbracket = [1, 2] - [1, 2] = 0$.
\end{proof}

We now conclude with the proof of Theorem \ref{thm:ordertype}.

\begin{proof}[Proof of Theorem \ref{thm:ordertype}]
For $i < i'$, we have that $\llbracket K_{(i, j)} \rrbracket \ll \llbracket K_{(i'\!, j')} \rrbracket$ by Lemma \ref{lem:order i}. For $j < j'$, we have that $\llbracket K_{(i, j)} \rrbracket \ll \llbracket K_{(i, j')} \rrbracket$ by Lemma \ref{lem:order i} ($i = 0$) and Lemma \ref{lem:order j} ($i > 0$). Thus,
\[ (i, j) < (i'\!, j') \implies \llbracket K_{(i, j)} \rrbracket \ll \llbracket K_{(i'\!, j')} \rrbracket. \]
It follows that $\{H_{\llbracket K_{(i, j)} \rrbracket} \mid (i, j) \in S\}$ is a filtration on $\cF$ with
\[ H_{\llbracket K_{(i, j)} \rrbracket} \subset H_{\llbracket K_{(i'\!, j')} \rrbracket} \quad \text{if } (i, j) < (i'\!, j') \]
and $\Z \subset H_{\llbracket K_{(i'\!, j')} \rrbracket}/H_{\llbracket K_{(i, j)} \rrbracket}$, generated by $\llbracket K_{(i'\!, j')} \rrbracket$. Recall the map $\phi: \cC \to \cF$. Letting
\[ \cF_{(i, j)} := \phi^{-1}\big[ H_{\llbracket K_{(i, j)} \rrbracket} \big], \]
we pull back to a filtration $\{\cF_{(i, j)} \mid (i, j) \in S\}$ on $\cC$ with $\Z \subset \cF_{(i'\!, j')}/\cF_{(i, j)}$ for $(i, j) < (i'\!, j')$.
\end{proof}

\bibliographystyle{amsalpha}
\bibliography{mybib}

\end{document}